\documentclass[a4paper,10pt,reqno]{amsart}
\usepackage{amssymb,latexsym,bbm,amsmath,amsfonts}
\usepackage{amsthm}
\usepackage[mathscr]{eucal}
\usepackage{rotating}
\usepackage{multirow}

\numberwithin{equation}{section}
\newtheorem{theorem}{Theorem}
\newtheorem{lemma}{Lemma}

\theoremstyle{definition}

\theoremstyle{remark}

\begin{document}
	
	\title[Asymptotic behaviour of pantograph equations]
	{Characterisation of asymptotic behaviour of perturbed deterministic and stochastic pantograph equations}
	
	\author{John A. D. Appleby}
	\address{School of Mathematical
		Sciences, Dublin City University, Glasnevin, Dublin 9, Ireland}
	\email{john.appleby@dcu.ie} 
	
	\author{Emmet Lawless}
	\address{Department of Mathematics, 
		University of Michigan, 
		Ann Arbor, MI 48109-1043, U.S.A.}
	\email{elawless@umich.edu}
	
	\thanks{JA is supported by the RSE Saltire Facilitation Network on Stochastic Differential Equations: Theory, Numerics and Applications (RSE1832). EL was supported by Science Foundation Ireland (16/IA/4443).} 
	\date{24 September 2025}
	
	\begin{abstract}
This paper considers the asymptotic behaviour of  deterministically and stochastically forced linear pantograph equations. The asymptotic behaviour is studied in the case when all solutions of the  pantograph equation without forcing tend to a trivial equilibrium. 
In all  cases, we  give necessary and sufficient conditions on the forcing terms which enable all solutions to converge to the equilibrium of the unforced equation, and which enable solutions to remain bounded. In the deterministic case, we give sharp conditions on forcing terms which enable the solutions of the forced equations to inherit the power law behaviour of the unforced equation, as well as slower rates of decay or growth not present in the unforced equation. Extensions to equations with general unbounded delay, and to finite--dimensional equations are also presented.
	\end{abstract}
	
	\maketitle
	
	\section{Motivation and Scope of the Work}
	\label{sec:2}
	
	Let $a\neq 0$, $b$ be real numbers, and let $q\in (0,1)$. The (scalar valued solution of the) delay differential equation  
	\begin{equation}\label{eq.unforced}
		u'(t)=bu(t)+au(qt), \quad t\geq 0; \quad u(0)\neq 0,
	\end{equation}
	has attracted a lot of attention. The is a so--called pantograph equation, first studied in depth by Kato and McLeod~\cite{KatoMc} and Fox et al and Ockendon and Tayler~\cite{Foxetal,OckTay}, with finite dimensional analysis in \cite{CarrD1,CarrDyson2}. Apart from interest in applications, it attracts mathematical interest as it is an example of a delay differential equation with unbounded delay, and, despite the constant multipliers $a$ and $b$, the equation is also non--autonomous, by virtue of the ``proportional'' delay $qt$. One interesting feature of linear differential equations with unbounded delay is that solutions that tend to an equilibrium need not do so at an exponentially fast rate (see e.g., \cite{AppBuck10,Kris88,MakTerj,Diblik98,DibRuz}). Pantograph equations are no exception, and Kato and McLeod \cite{KatoMc} showed that solutions which tend to zero do so at a power law rate of decay. Further treatments of deterministic equations with unbounded delay include \cite{cermak, HaddKris, Iser, Liu}, and work continues to date e.g. \cite{cer2022}. Fractional and integrodifferential equations of pantograph type have also been studied \cite{MR4534554}.
	
	There has been a lot of activity in extending results to stochastically perturbed pantograph equations and equations with unbounded delay, especially in the direction of asymptotic behaviour and numerical methods (see e.g.\cite{App2005, BakBuck2000,BakBuck2005, Fan,Huang,Huetal,Jiang,LiuChen,Meng1,Meng2,PJ2012, YanHuang,Yangetc}). The first such work on asymptotic behaviour is \cite{AppBuck2016}. In this field of stochastic asymptotic analysis, papers concerning power--law behaviour and numerical methods \cite{MR4439350,MR3927238,Milo}, large deviations \cite{MR4585696}, and neutral equations \cite{MR4278069} and L\'evy noise \cite{MR4053553} have also been recently considered. 
	Many \textit{sufficient conditions} have been obtained with respect to which solutions of stochastic equations exhibit convergence to a limit, either in $p$--th moment or pathwise. However, what appears to be absent in this literature are results which \textit{characterise} certain types of asymptotic behaviour in terms of forcing terms, such as, for instance, convergence to the equilibrium of the unforced equation, or asymptotic rates of growth or decay of forced solutions; it is one of the chief contributions of this paper to supply such asymptotic characterisations. But as we indicate now, we believe that the techniques in this paper have much wider application than to pantograph equations. 
	
	In this paper, we consider perturbations of \eqref{eq.unforced} where the intensity of the forcing terms is state--independent. Both deterministic and stochastic equations are considered, and in the case of stochastic equations, we allow the solution to be forced both deterministically and stochastically. The main contribution of the work is to give a  \textit{characterisation} of certain convergence, boundedness, and growth estimates in terms of the forcing terms. By a  characterisation of an asymptotic property \textbf{A} of solutions, we mean that solutions possess property \textbf{A} holds if and only if some conditions \textbf{C} are satisfied by the forcing terms; moreover, the conditions \textbf{C} should be expressible without recourse to properties of the solution of the forced equation, or other auxiliary objects that cannot be directly computed from the problem data. Ideally, such conditions should be easy to check. The characterisations we give satisfy these natural requirements.   
	
	The techniques used in this paper are especially suited to considering such ``additive'' rather than ``multiplicative'' perturbations (the terminology and analysis is more commonplace for stochastic equations, but is appropriate for deterministic equations too). Stochastic equations with such state--dependent noise have been studied in the papers listed above. Our analysis of additive noise equations relies on  a decomposition approach to solutions of stochastic delay differential equations, which allows such solutions to be built from solutions of stochastic ordinary differential equations with non--differentiable paths, and solutions of random functional differential equations which are continuously differentiable. A similar approach can be applied to equations with multiplicative noise, and a sketch of possible research in that direction is given in Section 6.  
	
	Other generalisations, to scalar equations with a general unbounded delay, or to multidimensional equations, are considered in Section 6, and sketch proofs are given which show that at least basic convergence and boundedness results still hold, giving reassurance that the theory is not confined to the narrow class of scalar, proportional delay equations. However, the main thrust of our paper concerns pantograph equations, partly due to their celebrated role in the historical development of the theory of unbounded delay equations, and partly because, since the asymptotic behaviour of the unperturbed pantograph equation is essentially perfectly understood (in contrast to many other unbounded delay equations) we can exploit this knowledge to obtain equally pristine results in the perturbed case, thereby demonstrating the sharpness of our approach. 
	
	In doing so, we are able to improve decay estimates on perturbed equations, show that these estimates are optimal in certain cases, accurately characterising when perturbation sizes cause solutions to transition from preserving the asymptotic behaviour of the unperturbed equation to slower rates of decay (or even growth) when the perturbation are larger than a critical size. In fact, by proving both big-$O$ and little-$o$ growth estimates on solutions throughout, we are able to characterise in many cases (in terms of an easily estimated functional of the perturbation, which is entirely independent of the structure of the unperturbed equation) when the solution of the perturbed equation has a given order of growth (or decay). The combination of (a) sharp estimation of growth rates (b) corresponding  correct determination of the transition point from the small to the large perturbation regime, and (c) characterisation of the growth conditions on the forcing terms which give rise to certain rates of growth or decay in solutions, are to the best of our knowledge new in the asymptotic study of unbounded delay differential equations. 
	
	Another novel feature in our analysis is to be able, for functional differential equations which are simultaneously deterministically and stochastically perturbed, to characterise certain important asymptotic behaviours. In particular, we are able to characterise exactly when solutions of perturbed equations are convergent or bounded, and the characterising conditions are very easy to check. Such a characterisation, with only continuity side conditions on forcing terms being required, seems new to the literature, \textit{even for non--delay stochastic differential equations}. Better yet, we sketch at the end of Section 5 an argument which suggests strongly that \textit{such a characterisation might well apply to a very broad class of functional differential equations} which can have finite or unbounded delay, and which does not possess any of the special structure of the pantograph equations which we study here. 
	
	We note that a great deal of the sharpness achieved in our results comes from a very careful asymptotic characterisation of solutions of scalar linear \textit{ordinary} and stochastic differential equations. Remarkably, given the simplicity of the equations involved, such a characterisation does not seem to be present in the literature, and part of the purpose of this paper is to show that a thorough \textit{development of this theory for linear ordinary differential equations will unlock a characterising perturbation theory for a very wide class of functional differential equations}. 
	
	We have chosen perturbed pantograph equations to do this, because the perturbation theory of non--autonomous delay differential equations (of which the pantograph equation is a well--studied example) has traditionally found comparatively little use for resolvents or fundamental solutions, at least compared with autonomous equations. Indeed, the analysis in this paper likewise makes no use whatsoever of resolvents or variation of constants formulae. In itself, this is neither noteworthy nor surprising. However, in other works, the authors have given asymptotic characterisation results for perturbations of autonomous deterministic and stochastic functional and Volterra equations, in which deterministic and stochastic variation of constants formulae are in constant use. This might leave one to believe that such formulae play a critical role in making the perturbation theory both achievable and sharp. \textit{However, this paper shows that while these resolvent--based representations are certainly useful when available, they do not appear to be essential for the development of a sharp perturbation theory}. As a result, the pantograph equation represents an interesting test case for a sharp perturbation theory that applies to other non--autonomous or even non--linear functional differential equations.         
	
	\subsection{Technical introduction}
	To make some of this general discussion precise, we let $f$ and $\sigma$ be continuous deterministic functions, $B$ be a standard Brownian motion, and $a$ and $b$ be real numbers obeying 
	\begin{equation} \label{eq.unforcedasstable}
		b<0,  \quad |b|>|a|,
	\end{equation}
	and consider 
	\begin{gather*}
		x'(t)=ax(qt)+bx(t)+f(t),\\
		dX(t)=(aX(qt)+bX(t)+f(t))\,dt+\sigma(t)\,dB(t).
	\end{gather*}
	\eqref{eq.unforcedasstable} is a  necessary and sufficient condition for all solutions of \eqref{eq.unforced} to tend to zero as $t\to\infty$ \cite{KatoMc}. Throughout we take $a\neq 0$, since this case deals with non--delayed equations. In the case of the deterministic equation, we give necessary and sufficient conditions on $f$ for which solutions will tend to zero, or for which solutions are bounded. 
	We also characterise when solutions $x$ are unbounded with a certain order of growth, or decay to zero more slowly than those of \eqref{eq.unforced}, and give sharp conditions under which the decay rate of $x$ is the same as that of \eqref{eq.unforced}. For the stochastic equation, we give conditions on $f$ and $\sigma$ which characterise when $X(t)\to 0$ as $t\to\infty$ a.s. and when $X$ is almost surely bounded. Results for multi--dimensional equations are also sketched. In later works, we will develop quantitative pathwise estimates for $X$, and relax the second inequality in \eqref{eq.unforcedasstable}, so as to study the situation when the underlying unforced equation \eqref{eq.unforced} has unbounded solutions. 
	
	In short, therefore this paper shows that provided all solutions of \eqref{eq.unforced} tend to zero, we can find coincident necessary and sufficient conditions on the forcing terms $f$ and $\sigma$ (which do not depend on the parameters $b$, $a$ and $q$) under which the solutions are convergent to zero; in particular, it is not necessary to strengthen conditions on the underlying equation in order to preserve stability, merely to impose appropriate smallness conditions on $f$ and $\sigma$. 
	
	The results in this paper develop work of the authors for certain types of \textit{uniformly asymptotically stable} and \textit{autonomous} linear differential and Volterra equations which are forced both deterministically and stochastically by time--independent forcing terms.  The nature of conditions we impose on $f$ and $\sigma$ for pantograph equations are identical to those needed in the autonomous case, and we avail of existing results from \cite{AppLaw2024ODE} for ordinary differential equations in our proofs. 
	Related results for the mean square asymptotic behaviour of affine equations, where there is also state--dependence in the diffusion coefficient, are given in  \cite{AL:2023(AppliedNumMath)}. For the stochastic differential equation 
	$
	dY_0(t)=-Y_0(t)\,dt +\sigma(t)\,dB(t)$ it is shown in \cite{ACR:2011(DCDS)} that  $Y_0(t)\to 0$ as $t\to\infty$ a.s.  if and only if $S(\epsilon)<+\infty$ for all $\epsilon>0$, where 
	\begin{equation} \label{eq.S}
		S(\epsilon):= \sum_{n=1}^\infty \sqrt{\int_{n-1}^n\sigma^2(s)\,ds} \exp\left(-\epsilon\int_{n-1}^{n}\sigma^2(s)\,ds\right).
	\end{equation}
	In \cite{AppLaw2024ODE} (as discussed in Section 4), it is shown that solutions of the ordinary differential equation $y'(t)=-y(t)+f(t)$ obey $y(t)=O(\gamma(t))$ as $t\to\infty$ for non--decreasing and so--called subexponential weight functions $\gamma$ if and only if $f_\theta$, defined by  
	$f_\theta(t):=\int_{(t-\theta)^+}^t f(s)\,ds$ for $(t,\theta)\in [0,\infty)\times[0,1]$, obeys
	\begin{equation} 
		\label{eq.fthetagammabdd}
		\sup_{0\leq \theta\leq 1} |f_\theta(t)| \leq K\gamma(t), \quad t\geq 0, \quad \text{for some $K>0$}.
	\end{equation}
	(In the above, and throughout the paper, we make extensive use of the standard ``big O'' and ``little o'' Landau notation, as well as the notation $a(t)\sim b(t)$, $t\to\infty$ for real--valued functions $a$ and $b$ to signify that $a(t)/b(t)\to 1$ as $t\to\infty$).
	Finally, in this volume, we also show that the asymptotic behaviour of $f_\theta$  characterises the situation when solutions of deterministic functional and  Volterra differential equations have time averages (see \cite{AL:2024Cesaro}).  
	
	Therefore, we see that the asymptotic behaviour of $f_\theta$ for $\theta\in [0,1]$ and $\sigma^2_1(t):=\int_{(t-1)^+}^t \sigma^2(s)\,ds$ appear to characterise many interesting types of asymptotic behaviour of solutions of forced autonomous equations. In this work, we show that this is true for perturbed pantograph equations also. \textit{Therefore, the evidence accumulates in favour of the overall viewpoint that situations in which important qualitative and quantitative asymptotic properties of solutions of perturbed differential systems (whether deterministic or stochastic, memory--dependent or not, autonomous or not, and perhaps even linear or not), may be characterised by in terms of perturbations through $\{f_\theta:\theta\in [0,1]\}$ and $\sigma^2_1$}.  
	
	%
	
	We recall some facts about the asymptotic behaviour of solutions of  \eqref{eq.unforced}; all results appear in \cite{KatoMc}. There is no loss of generality in taking the initial condition to be unity; denoting the solution of the differential equation with initial condition $\xi\in \mathbb{R}$ by $u(\cdot;\xi)$, we have that 
	$u(t;\xi)=u(t;1)\xi$ for all $t\geq 0$. Notice that if $\xi=0$, then $u(t;\xi)=0$ for all $t\geq 0$, so that the zero function is an equilibrium solution for all choices of the parameters (when $a+b=0$, every constant is a solution, but we do not study this phenomenon). 
	
	If $b>0$, $u(t)/e^{bt}$ tends to a finite limit, so solutions exhibit real exponential growth: we will not study this case here. When $b<0$, however, power--law asymptotic behaviour occurs. Specifically, since $a\neq 0$ and $b<0$ we may define 
	\begin{equation} \label{eq.kappa}
		\kappa:=-\frac{\log|b/a|}{\log(1/q)},
	\end{equation}
	and $u(t)=O(t^\kappa)$ as $t\to\infty$. Indeed, there are no non--trivial solutions that are $o(t^\kappa)$ as $t\to\infty$, so this estimate is very sharp \cite[Theorem 3]{KatoMc}. More precise information regarding oscillation of solutions are known, but since we focus on sharp upper estimates of solutions of perturbed equations, we do not seek such refined results. 
	
	This discussion explains in part why we wish to confine attention in this paper to the case when \eqref{eq.unforcedasstable} holds, but does not mean that it is not interesting, or impossible, to give a perturbation theory in the other parameter regimes mentioned above, and it is a task that we will return to in other works.

	\section{A Fundamental Estimate and Large Perturbation Estimates}
	
	We study the deterministic pantograph equation first, partly for reasons of simplicity, but also because solutions of the stochastic equation can be written in terms of the deterministic equation, when viewed path--by--path in the sample space. Therefore the analysis of the deterministic equation is fundamental. Suppose that 
	\begin{equation} \label{eq.fcns}
		f\in C([0,\infty);\mathbb{R}),
	\end{equation} 
	and consider the forced pantograph equation
	\begin{equation} \label{eq.forced}
		x'(t)=ax(qt)+bx(t)+f(t), \quad t\geq 0; \quad x(0)=\zeta\in \mathbb{R}.
	\end{equation}
	There is a unique continuous solution to this initial value problem (see e.g.~\cite{AppBuck10}). If \eqref{eq.unforcedasstable} holds, then  $u(t)\to 0$ as $t\to\infty$. It is therefore reasonable to ask under what ``smallness'' conditions on the forcing function $f$ certain asymptotic properties (such as convergence to zero, or rates of decay) of $u$ are preserved; on the other hand, we may be interested in understanding the properties of solutions when $f$ does not obey smallness   properties which preserve properties of $u$.  
	
	To this end, in this section, we establish a ``fundamental'' estimate on the size of solutions of perturbed pantograph equations which enables us to obtain not only convergence and stability results, but also (after further analysis later in the paper) very sharp asymptotic bounds, especially when perturbations are, in a certain sense, small.
	The results are of a ``sufficiency'' type, which is to say, if sufficiently strong asymptotic conditions are imposed on the forcing term, this will generate a rate of decay on the solution. A full characterisation of asymptotic behaviour of solutions will require ``necessity'' results, as well as some transformations and decomposition results; this is because our fundamental lemma is usually applied to a perturbed pantograph equation which results from such a decomposition, rather than to the original pantograph equation itself. We will also need estimates which deal with large perturbations, and these will be proved separately. Therefore, our fundamental estimate, although a key tool in the asymptotic analysis, provides only part of the picture. 
	
	This approach differs from our analysis of autonomous differential and Volterra integrodifferential equations, where we exploit variation of constants formulae and known (integrable) behaviour of the differential resolvent. Although perturbed pantograph equations admit such variation of constants formulae, the non--autonomous character of the equation makes it difficult to get good asymptotic information about the underlying resolvent, and this motivates an approach avoiding this problem. In this sense, our approach bears relation to that taken in \cite{AppBuck10}.      
	
	In order to state our fundamental estimate, we introduce some notation. For $a\neq 0$, there are (in general, complex valued) solutions $k$ to the equation $aq^k+b=0$. If $k_0\in \mathbb{C}$ is such a solution,  $\kappa:=\text{Re}(k_0)$ is uniquely defined, and given by \eqref{eq.kappa}.
	Let 
	\begin{align} \label{eq.c}
		c&:=\log q<0,\\
		\label{eq.alpha}
		\alpha&:=|a/b|,
	\end{align}
	so that $e^c=q$, and under \eqref{eq.unforcedasstable}, $\alpha\in (0,1)$. 
	
	\begin{lemma} \label{lemma.fundamental}
		Suppose $q\in (0,1)$ and \eqref{eq.unforcedasstable} holds. Let $\varphi\in C([0,\infty);\mathbb{R})$ and $z$ be the unique continuous solution to 
		\begin{equation} \label{eq.z}
			z'(t)=az(qt)+bz(t)+\varphi(t),\quad t\geq 0; \quad z(0)=\xi\in\mathbb{R}.
		\end{equation}
		Let $\alpha$, $c$ be defined by \eqref{eq.alpha}, \eqref{eq.c} and define $I_n:=[s_0-nc,s_0-(n+1)c]$. Let  
		\begin{equation} \label{eq.Kn}
			K_n:=\sup_{s\in I_n} e^{-\kappa s}|z(e^s)|,
		\end{equation}
		and let $\epsilon_n>0$ be a sequence such that  
		\begin{equation} \label{eq.epsn}
			\sup_{s\in I_{n+1}} |\varphi(e^s)|\leq \epsilon_{n+1}.
		\end{equation}
		\begin{enumerate}
			\item[(i)] If $(\epsilon_n/\alpha^n)$ is summable, $K_n\leq K^\ast$, and  $z(t) = \text{O}(t^\kappa)$ as $t\to\infty$.  
			\item[(ii)] If $(\epsilon_n/\alpha^n)$ is not summable, then $
			K_n\leq K^\ast \sum_{j=1}^n \frac{\epsilon_j}{\alpha^j}$
			and 
			\begin{equation} \label{eq.fundest2}
				\sup_{t\in [e^{s_0-nc},e^{s_0-nc}e^{-c}]} |z(t)|\leq K^\ast  \alpha^n \sum_{j=1}^n \frac{\epsilon_j}{\alpha^j}.
			\end{equation}
		\end{enumerate}
	\end{lemma}
	
	The proof of the result is inspired especially by that of Kato and McLeod~\cite[Theorem 3(i)]{KatoMc} and also of Cermak~\cite[Lemma 3.3]{cermak}. It differs from the former in considering perturbed equations. Concerning the latter paper and result, we note that it  considers more general perturbed delay equations, and yields estimates on their solutions using \textit{pointwise} conditions on forcing terms. However, the summability conditions in Lemma~\ref{lemma.fundamental} hint to the development of \textit{integrability} conditions on forcing terms in our analysis, which will ultimately yield extra sharpness to our results. In addition, our decomposition approach to solutions of the perturbed pantograph equations, shown in the next section, has the effect of turning pointwise conditions on the forcing terms into average conditions too, which will further sharpen our deterministic results. The decomposition approach will even enable us  exploit the fundamental lemma for stochastic equations, where (poor) pointwise asymptotic behaviour and regularity (i.e. non--differentiability)  of the forcing terms would confound the analysis in \cite{cermak}.    
	
	At this instant, the reader must take these claims on trust, since it is not immediately apparent how sharp or useful the estimate in part (ii) is, nor that it can be applied  to study stochastic equations. In fact, quite a lot of further analysis is needed so that  full value can be extracted from this lemma. Nevertheless, we will eventually show that this fundamental lemma (a) develops essentially unimprovable size estimates on the size of solutions of deterministic equations, which improve all existing asymptotic estimates; (b) allows us to give characterisations of the almost sure convergence and boundedness of stochastically perturbed equations, results hitherto absent from the literature. 
	
	We turn to these matters later. 
	However, we can immediately prove two direct corollaries of Lemma~\ref{lemma.fundamental}, under the assumption \eqref{eq.unforcedasstable}, giving some immediate reassurance that the Lemma~\ref{lemma.fundamental} will be of some utility. 
	
	\begin{theorem} \label{thm.1}
		Suppose $q\in (0,1)$ and \eqref{eq.unforcedasstable} holds. Let $\varphi\in C([0,\infty);\mathbb{R})$ and $z$ be the unique continuous solution to \eqref{eq.z}. 
		\begin{enumerate}
			\item[(a)] If $\varphi(t)\to 0$ as $t\to\infty$, then $z(t)\to 0$ as $t\to\infty$. 
			\item[(b)] If $\varphi$ is bounded, then $z$ is bounded. 
		\end{enumerate}		
	\end{theorem}
	\begin{proof}
		If $\varphi(t)\to 0$ as $t\to\infty$, we may take equality in \eqref{eq.epsn}, and note the hypothesis forces $\epsilon_n\to 0$ as $n\to\infty$. By \eqref{eq.unforcedasstable}, $\alpha\in (0,1)$. Either $\lambda_n=\epsilon_n/\alpha^n$ is summable or not. 
		
		Suppose $\lambda_n$ is summable. Then, by Lemma~\ref{lemma.fundamental} (i),  
		$\sup_{t\in [e^{s_0-nc},e^{s_0-nc}e^{-c}]} |z(t)|\leq K^\ast [e^{(s_0-nc)}]^\kappa$,
		so $z(t)\to 0$ as $t\to\infty$, because $\kappa<0$ and $c<0$. 
		
		Suppose $\lambda_n$ is not summable. Since $(1/\alpha^n)$ is divergent, by Toeplitz lemma (see \cite[p.390]{Shiryaev}), 	$\sum_{j=1}^n \epsilon_j \alpha^{-j} = o\left(\sum_{j=1}^n \alpha^{-j}\right)$ as $n\to\infty$, since $\epsilon_n\to 0$ as $n\to\infty$. On the other hand $\sum_{j=1}^n \alpha^{-j} \sim \alpha^{-n} (1-\alpha^{-1})$ as $n\to\infty$, so $\alpha^n\sum_{j=1}^n \epsilon_j \alpha^{-j}\to 0$ as $n\to\infty$.
		Since the righthand side of \eqref{eq.fundest2} tends to $0$ as $n\to\infty$, $z(t)\to 0$ as $t\to\infty$.  
		
		In the case when $\varphi$ is bounded on $[0,\infty)]$ by $\overline{\varphi}$, say, we may fix $\epsilon_n=\overline{\varphi}$. By part (i) of the lemma, which holds whether $(\epsilon_n/\alpha^n)$ is summable or not, we have that 
		$		K_n\leq K^\ast \overline{\varphi} \sum_{j=1}^n \alpha^{-j} \leq \bar{K} \alpha^{-n}$ 
		since $\alpha<1$, for some $\bar{K}>0$. Thus $K_n \alpha^n\leq \bar{K}$. Revisiting the above estimates gives  
		\[
		\sup_{t\in [e^{s_0-nc},e^{s_0-nc}e^{-c}]} |z(t)|\leq K_n \alpha^n  e^{\kappa s_0}
		\leq \bar{K}e^{\kappa s_0}=:\bar{z},
		\]
		and so $z$ is bounded, as required. 
	\end{proof}
	As we see in Section 6, this qualitative result generalises  to equations with unbounded delay, using methods from \cite{AppBuck10} (see also  Theorem~\ref{thm.monotonepert} below). But our primary interest in the fundamental Lemma derives from quantitative estimates that it generates, stated below. For these estimates, we recall that a (measurable) function $\gamma:[0,\infty)\to (0,\infty)$ is regularly varying at infinity with index $\alpha$ if $\lim_{t\to\infty} \gamma(\lambda t)/\gamma(t)=\lambda^\alpha$ for all $\lambda>0$. The notation $\gamma\in \text{RV}_\infty(\alpha)$ is used. We will exploit properties of regularly varying functions throughout this work; see Bingham, Goldie and Teugels \cite{BGT}, especially Chapter 1, for relevant results.    
	\begin{theorem} \label{thm.growthbounds2}
		Suppose that $q\in (0,1)$ and \eqref{eq.unforcedasstable} holds. Let $\varphi\in C([0,\infty);\mathbb{R})$ and $z$ be the unique continuous solution to \eqref{eq.z},  Let $\varphi(t)=\text{O}(\gamma(t))$ as $t\to\infty$, where $\gamma$ is regularly varying at infinity. Then   
		\[
		z(t)=\text{O}\left(t^\kappa\int_1^t \frac{\gamma(s)}{s^{\kappa+1}}\,ds\right), \quad t\to\infty.
		\] 
	\end{theorem} 
	The proof is deferred to the end. Since the theorem gives only bounds on solutions, it is reasonable to ask whether (i) the bounds are sharp; (ii) the transition from the rate of decay of the  solutions of \eqref{eq.unforced} to slower decay rates, do indeed occur when $t\mapsto\varphi(t)/t^{\kappa+1}$ changes from being an integrable to a non--integrable function. The following example shows that Theorem \ref{thm.growthbounds2} is sharp in both regards. Furthermore, in Section 4, we give another example which shows that the ``big O'' estimate cannot be improved in general to results of the form $z(t)\sim c t^\kappa$ as $t\to\infty$, at least in the case of small perturbations (i.e. when $t\mapsto \gamma(t)/t^{\kappa+1}$ is integrable).     
	\begin{theorem} \label{thm.growthbdssharp}
		Let $a>0$, $b<0$, $\varphi\in C(\mathbb{R}^+)$ and $z$ the continuous solution to \eqref{eq.z}.
		\begin{enumerate}
			\item[(a)] For every $C>0$ and every  $\gamma\in \text{RV}_\infty(\kappa)$ such that $t\mapsto \gamma(t)/t^{\kappa+1}$ is not in $L^1(0,\infty)$, there exists $\varphi(t)\sim C\gamma(t)$ such that for some $D>0$,
			\[
			z(t)\sim Dt^\kappa\int_1^t \frac{\gamma(s)}{s^{\kappa+1}}\,ds, \quad t\to\infty.
			\] 
			\item[(b)] For every $C>0$ and every  $\gamma\in \text{RV}_\infty(\kappa)$ such that $t\mapsto \gamma(t)/t^{\kappa+1}$ is in $L^1(0,\infty)$, there exists $\varphi(t)\sim C\gamma(t)$ such that $z(t)\sim Dt^\kappa$ as $t\to\infty$ 
			for some $D>0$.
		\end{enumerate}
	\end{theorem}  
	\begin{proof}
		For part (a), write $\psi(t):=\gamma(t)/t^{\kappa+1}$. Then $\psi\in \text{RV}_\infty(-1)$ and $\psi$ is not integrable. We may take $\psi$ to be a smooth regularly varying function. Let $\Psi(t):=\int_{1/q}^t \psi(s)\,ds$.
		Then $\Psi$ is regularly varying with index zero, and $\Psi(t)\to \infty$ as $t\to\infty$. Moreover, $\Psi$ is smoothly regularly varying, with $t\Psi'(t)/\Psi(t)\to 0$ as $t\to\infty$. Set $z(t):=Dt^\kappa \Psi(t)\,t\geq q$ and $\varphi(t):=z'(t)-bz(t)-az(qt), \,t\geq 1$. Then $z$ satisfies \eqref{eq.forced} on $[1,\infty)$. Moreover  
		$z'(t)=Dt^{\kappa-1}\Psi(t)(\kappa+t\Psi'(t)/\Psi(t))\sim \kappa Dt^{\kappa-1}\Psi(t)$ as $t\to\infty$. 	For $t\geq 1$, using the fact that $b+aq^\kappa=0$, we have that
		$bz(t)+az(qt) = bDt^\kappa(\Psi(t)-\Psi(qt))$. Compare the asymptotic behaviour of $z'$ and the last term, to get the dominant asymptotic behaviour in $\varphi$.  $\Psi'=\psi$ is regularly varying with index $-1$, the uniform convergence theorem for regularly varying functions gives 
		\[
		\frac{\Psi(t)-\Psi(qt)}{t\Psi'(t)}=\int_q^1 \frac{\Psi'(\lambda t)}{\Psi'(t)}\,d\lambda
		\to \int_q^1 \lambda^{-1}\,d\lambda = \log(1/q), \quad t\to\infty.
		\]
		Hence 
		$bz(t)+az(qt) \sim b \log(1/q) D t^{\kappa+1} \psi(t)$ as $t\to\infty$, 
		so this expression is in $RV_{\infty}(\kappa)$, while $z'$ is in $RV_{\infty}(\kappa-1)$. Therefore $z'(t)=o(bz(t)+az(qt))$, and so  we have
		$\varphi(t)\sim -(bz(t)+az(qt)) \sim C t^{\kappa+1} \psi(t)$, where $C=-b \log(1/q) D>0$.
		
		For part (b), let $\psi$ be a smooth regularly varying function with index $-1$. Then $\Psi$ is regularly varying with index zero, but now since $\psi$ is integrable $\Psi(t)\to L\in (0,\infty)$ as $t\to\infty$. Moreover, $\Psi$ is also smoothly regularly varying, and so $t\Psi'(t)/\Psi(t)\to 0$ as $t\to\infty$. Next let $D>0$ be arbitrary and set
		\[
		z(t)=\frac{D}{L+1}\left(t^\kappa + t^\kappa\int_0^t \psi(s)\,ds\right), \quad t\geq q.
		\]
		Set $\varphi(t)=z'(t)-(bz(t)+az(qt))$ for $t\geq 1$. Then $z$ is the unique continuous solution of \eqref{eq.forced} on $[1,\infty)$ and $z(t)\sim Dt^\kappa$ as $t\to\infty$. We determine the asymptotic behaviour of $f$. Write $\tilde{D}=D/(L+1)$. To start, note $z'(t)=\tilde{D}\kappa t^{\kappa-1}\left(\kappa+ \kappa \Psi(t) +  t\psi(t)\right)$. 
		Since $t\psi(t)=t\Psi'(t)/\Psi(t)\cdot \Psi(t)\to 0\cdot L$=0 as $t\to\infty$, $z'(t)$ is
		order $t^{\kappa-1}$ as $t\to\infty$. Next
		\[
		bz(t)+az(qt) = b\tilde{D} t^\kappa \int_{qt}^t \psi(s)\,ds,
		\]
		so, proceeding as before we get 
		$bz(t)+az(qt) \sim b \tilde{D}t^{\kappa+1} \psi(t)\log(1/q)$ as $t\to\infty$, which is regularly varying with index $\kappa$, and dominates $z'(t)$. Letting  $C:=-bD(L+1)^{-1}\log(1/q)>0$, gives  $\varphi(t)\sim C t^{\kappa+1} \psi(t)=C\gamma(t)$ as $t\to\infty$, as needed.  
	\end{proof}
	Perturbations with decay faster than a power law  are covered by Theorem~\ref{thm.growthbounds2}: if $\varphi(t)=o(t^{-M})$ for all $M>0$, then $z(t)=O(t^\kappa)$, by taking $\gamma(t)=t^{\kappa-2}$. 
	
	For larger perturbations that may not have a power--law type form, the following bound can be established directly, without recourse to the fundamental lemma. 
	
	\begin{theorem} \label{thm.monotonepert}
		Suppose that $q\in (0,1)$ and \eqref{eq.unforcedasstable} holds. Let $\varphi\in C([0,\infty);\mathbb{R})$ and $z$ be the unique continuous solution to \eqref{eq.z}. Let $\gamma$ be  $C^1$ and non--decreasing.
		\begin{enumerate}
			\item[(A)]
			If $\varphi(t)=\text{O}(\gamma(t))$ as $t\to\infty$, then $z(t)=\text{O}(\gamma(t))$, as $t\to\infty$. 
			\item[(B)] If $\varphi(t)=\text{o}(\gamma(t))$ and $\gamma(t)\to \infty$ as $t\to\infty$, then $z(t)=\text{o}(\gamma(t))$, as $t\to\infty$. 
		\end{enumerate}
	\end{theorem}
	\begin{proof} The proof is inspired by those in \cite{AppBuck10,AppBuck2016}. 
		Denote by 
		$D_+ u(t) = \limsup_{h\to 0^+} (u(t+h)-u(t))/h$ 
		the Dini derivative of any continuous function $u$. Since $b<0$, by considering the cases where $z(t)$ is positive, negative or zero, we arrive at the consolidated estimate 
		$D_+|z(t)|\leq b|z(t)|+ |a||z(qt)|+|\varphi(t)|$ for $t\geq 0$.
		We note that if $z_+$ is any positive $C^1(0,\infty)$ function satisfying $z_+(0)>|z(0)|$ and $z_+'(t)> bz+(t)+ |a|z_+(qt)+|\varphi(t)|$ for all $t\geq 0$, then $z_+(t)>|z(t)|$ for all $t\geq 0$, by standard comparison arguments. By hypothesis, there exists $C>0$ such that $|\varphi(t)|\leq C\gamma(t)$ for $t\geq 0$. Take $K>\max\left(C/(|b|-|a|), |z(0)|/\gamma(0)\right)$,
		and set $z_+(t)=K\gamma(t)$ for $t\geq 0$. Then $z_+(0)=K\gamma(0)> |z(0)|$ and for $t\geq 0$ we have 
		\begin{align*}
			bz_+(t)+|a|z_+(qt)+|\varphi(t)|&\leq bK\gamma(t)+|a|K\gamma(qt)+C\gamma(t)\\
			&\leq (bK+|a|K+C)\gamma(t)<0<K\gamma'(t)=z_+'(t),
		\end{align*}
		using the monotonicity of $\gamma$. Thus $|z(t)|< z_+(t)=K\gamma(t)$ for all $t\geq 0$, as needed. 
		
		For part (B), by hypothesis, we have $|\varphi(t)|<\epsilon \gamma(t), \, t\geq T_1(\epsilon)$. For $t\geq qT_1(\epsilon)$, define $z_\epsilon(t)=\epsilon \gamma(t)/(|b|-|a|)+C(\epsilon)$ 
		where $C(\epsilon):=1+\max_{t\in [qT_1,T_1]}|z(s)|>0$. Then $z_\epsilon(t)>|z(t)|$ for $t\in [qT_1,T_1]$. For $t\geq T_1$ we have
		\begin{align*}
			bz_\epsilon(t)+|a|z_\epsilon(qt)+|\varphi(t)|
			&=(b+|a|)C	+ \epsilon\left(\frac{b}{|b|-|a|}\gamma(t)+\frac{|a|}{|b|-|a|}\gamma(qt)\right)+|\varphi(t)|\\
			&\leq (b+|a|)C	+ \epsilon\frac{b+|a|}{|b|-|a|}\gamma(t)+\epsilon \gamma(t)\\
			&= (b+|a|)C<0<\frac{1}{|b|-|a|}\epsilon \gamma'(t)=z_\epsilon'(t).
		\end{align*}
		Thus $|z(t)|<z_\epsilon(t) = \epsilon \gamma(t)/(|b|-|a|)+C(\epsilon)$ for $t\geq T_1(\epsilon)$. Dividing by $\gamma(t)$ on both sides, letting $t\to\infty$, and then $\epsilon\to 0$, we get the desired result. 
	\end{proof}
	It is easy to show that these bounds are sharp in at least some cases by considering an example. Suppose that $\gamma'(t)/\gamma(t)\to 0$ as $t\to\infty$ and $t\gamma'(t)/\gamma(t)\to \infty$ as $t\to\infty$. Then $\gamma$ is increasing subexponentially (by the first limit, see \eqref{eq.psisub} and the discussion thereafter), but faster than any power, by the second. Thus $\gamma(qt)/\gamma(t)\to 0$ as $t\to\infty$. 
	Let $z(t)=D\gamma(t)$ for $t\geq 0$, and define 
	$\varphi(t)=D[\gamma'(t)-b\gamma(t)-a\gamma(qt)]$ for $t\geq 0$. Then $z$ obeys \eqref{eq.z}, $\varphi(t)\sim -bD\gamma(t)$ as $t\to\infty$, and $z(t)\sim D\gamma(t)$ as $t\to\infty$, confirming the sharpness in part (a). 
	
	The ``little--o'' result in part (B) of Theorem~\ref{thm.monotonepert} is less traditional, and the reader may query its  value. However, we will later show that such results can be used in conjunction with the conventional ``big--O'' results to characterise when solutions obey $\limsup_{t\to\infty}|z(t)|/\gamma(t)\in (0,\infty)$, thereby establishing the exact asymptotic order of the running maxima of solutions.  
	
	We note that Theorem~\ref{thm.monotonepert} \textit{does not} give results where an exact rate of growth of the solution is given e.g., $z(t)\sim c\gamma(t)$ as $t\to\infty$, where $\gamma(t)\to \infty$ as $t\to\infty$. To give conditions on the forcing terms under which this would occur is also of interest, and will form part of a future investigation.    

	\section{Deterministically forced equation: first results, auxiliary ODEs}
	Theorem \ref{thm.1} shows that $f(t)\to 0$ as $t\to\infty$ is a \textit{sufficient} condition in order for the solution $x$ of  \eqref{eq.forced} to obey $x(t)\to 0$ as $t\to\infty$. With a little extra work, it can be used to develop a necessary and sufficient condition. Consider  
	\begin{equation} \label{eq.y}
		y'(t)=-y(t)+f(t), \quad t\geq 0; \quad y(0)=0.
	\end{equation}  
	\eqref{eq.y} has a unique continuous solution, given by 
	$y(t)=\int_0^t e^{-(t-s)}f(s)\,ds$ for $t\geq 0$. 
	
	We now show that the asymptotic behaviour of \eqref{eq.forced} is determined solely by  $y$.
	\begin{theorem} \label{thm.detasymptoticstab}
		Suppose that $q\in (0,1)$ and \eqref{eq.unforcedasstable} holds. Let $f\in C([0,\infty);\mathbb{R})$ and $x$ be the unique continuous solution to  \eqref{eq.forced}. Then the following are equivalent
		\begin{enumerate}
			\item[(A)] $y(t)\to 0$ as $t\to\infty$;
			\item[(B)] For every $\theta>0$, $\lim_{t\to\infty} \int_t^{t+\theta} f(s)\,ds =0$;
			\item[(C)] $x(t)\to 0$ as $t\to\infty$. 
		\end{enumerate}
	\end{theorem}
	\begin{proof}
		The equivalence of (A) and (B) is known; we show (A) and (C) are equivalent.
		Suppose (A) holds. Set $z(t)=x(t)-y(t)$ for $t\geq 0$, $z(0)=x(0)=\zeta$. Define 
		\begin{equation} \label{def.varphi1}
			\varphi(t)=ay(qt)+(b+1)y(t), \quad t\geq 0. 
		\end{equation}
		Note that $\varphi$ is continuous. Since $z'=x'-y'$, we have that $z$ obeys \eqref{eq.z}. 
		If (A) holds, $\varphi(t)\to 0$ as $t\to\infty$. Since $z$ obeys \eqref{eq.z}, we may apply Theorem \ref{thm.1} part (a) to it, so $z(t)\to 0$ as $t\to\infty$. But $y(t)\to 0$ as $t\to\infty$ and $x=y+z$, so (C) follows. 
		
		To show (C) implies (A), re--write \eqref{eq.forced} according to  
		$x'(t) = -x(t) + \{x(t)+bx(t)+ax(qt)\} + f(t)$ for $t\geq 0$. Let $g(t)$ be the sum of the last two terms; then $x'=-x+g$ and thus integrating by parts and rearranging we get 
		\begin{equation} \label{eq.yrepx}
			y(t)=x(t)-\xi e^{-t} - \int_0^t e^{-(t-s)} \{(1+b)x(s)+ax(qs)\}\,ds, \quad t\geq 0.
		\end{equation}
		Since (C) holds, the first two terms on the righthand side of \eqref{eq.yrepx} tend to zero as $t\to\infty$. By (C) the term in curly brackets in the integrand tends to zero as $s\to\infty$. The integral in \eqref{eq.yrepx} is  the convolution of an $L^1(0,\infty)$ function and a function in $BC_0(0,\infty)$, so it tends to $0$ as $t\to\infty$. Thus $y(t)\to 0$ as $t\to\infty$, proving (A).    
	\end{proof}
	We note that this appears to be the first instance in the literature of a result which \textit{characterises} conditions on perturbations which preserve the asymptotic stability of non--autonomous functional differential equations. The same remark can be made concerning boundedness; see Theorem~\ref{thm.detasymptoticbound} later. Sufficient conditions for which convergence  to the equilibrium are preserved, which already appear in the literature, include $f(t)\to 0$ as $t\to\infty$ or $f\in L^1(0,\infty)$ (we note that both imply $y(t)\to 0$ as $t\to\infty$), but Theorem~\ref{thm.detasymptoticstab} shows that neither are necessary. 
	
	This leads us to ask what functions $f$ allow $y(t)\to 0$ as $t\to\infty$ which do not satisfy the familiar stability conditions above. This matter has been explored somewhat in \cite{AL:2023(AppliedNumMath)}. Two classes of functions with bad ``pointwise'' behaviour, but good ``average'' behaviour naturally emerge, but we do not claim that these characterise all the functions with stability preserving properties. 
	
	The first class are functions which exhibit high frequency oscillation (even with unbounded amplitude, provided the frequency grows faster than the amplitude). A prototype for such a function is 
	\[
	f(t)= e^{\beta t}\sin(e^{\theta t}), \quad t\geq 0,
	\]  
	where $\theta>\beta>0$. This function, by most measures, is ill--behaved as $t\to\infty$, oscillating between exponentially growing upper and lower envelopes, with increasingly higher frequency. Obviously, $f(t)=O(e^{\beta t})$ as $t\to\infty$, yet $y(t)=O(\max(e^{-t},e^{-(\theta-\beta)t}))$ as $t\to\infty$. Therefore, as $y(t)\to 0$ as $t\to\infty$, we have that $x(t)\to 0$ as $t\to\infty$. 
	
	Another class of functions $f$ which give rise to convergence of $y$, and which do not satisfy $f(t)\to 0$ as $t\to\infty$, or $f\in L^1(\mathbb{R}^+)$, are nonnegative functions with unboundedly large ``spikes'' of very short duration. Notice that if $f$ is a non--negative function, the condition (B) in Theorem~\ref{thm.detasymptoticstab} is equivalent to 
	\[
	\int_n^{n+1} f(s)\,ds \to 0, \quad n\to\infty,
	\]  
	where $n$ tends to infinity through the integers. We construct such a function now; let let $n\in \mathbb{N}$, and suppose $0<w_n<1$ and $w_n\to 0$ as $n\to\infty$. Let $f$ be zero on $[n,n+1]$ except on the interval $[n+1/2-w_n/2,n+1/2+w_n/2]$. Let $f$ be linear on $[n+1/2-w_n/2,n+1/2]$, passing through the points $(n+1/2-w_n/2,0)$ and $(n+1/2,h_n)$ where $h_n>0$ and $h_n\to\infty$ as $n\to\infty$, and likewise let $f$ be linear on  $[n+1/2,n+1/2+w_n/2]$, passing through the points  $(n+1/2,h_n)$ and $(n+1/2+w_n/2,0)$. Then $f$ is continuous and nonnegative and obeys 
	\[
	\int_{n}^{n+1} f(s)\,ds  = \frac{1}{2}w_n h_n.
	\] 
	Clearly, $\max_{n\leq s\leq n+1} f(s) = h_n$, so if we arrange that $h_n\to \infty$ as $n\to\infty$, while $w_n h_n\to 0$ as $n\to\infty$, we have a function which satisfies condition (B), and obeys $\limsup_{t\to\infty} f(t)=+\infty$. Indeed, since $h_n$ can be chosen to grow arbitrarily quickly, the running maximum of $f$ can grow arbitrarily quickly in time. It is also easy to arrange that $f$ is not in $L^1(\mathbb{R}^+)$, since for any
	$N\in \mathbb{N}$ we have 
	\[
	\int_0^{N+1} |f(s)|\,ds = \frac{1}{2} \sum_{n=0}^N w_n h_n,
	\]
	so it is merely necessary to arrange for $w_n h_n$ to be divergent. Thus, we may have $f$ having partial maxima which grow arbitrarily fast (let us take $h_n\to \infty$ as $n\to\infty$ in any way we choose such that  $h_n>1$ and $h$ increasing); picking $w_n=\frac{1}{nh_n}$, we have that  $0<w_n<1$, $w_n\to 0$ as $n\to\infty$, condition (B) holds, but $\limsup_{t\to\infty} f(t)=+\infty$ and $f\not\in L^1(\mathbb{R}^+)$. Thus, for such $f$, we have that $y(t)\to 0$ as $t\to\infty$, and hence $x(t)\to 0$ as $t\to\infty$.   
	
	Theorem~\ref{thm.detasymptoticstab} motivates a characterisation of the asymptotic behaviour of the linear differential equation \eqref{eq.y} given in \cite{AppLaw2024ODE}, from which we will presently draw. The conditions and arguments in \cite{AppLaw2024ODE} are inspired by analysis of Strauss and Yorke \cite{SY67a,SY67b} and by Grippenberg, London and Staffans~\cite[Lemma 15.9.2]{GLS}.  
	
	We now discuss the connection between (A) and (B) in Theorem~\ref{thm.detasymptoticstab}. The fact that (B) implies (A) is due to \cite{GLS}, while \cite{SY67a,SY67b} established the necessity and sufficiency of (B), albeit with extra uniform (in $\theta$) convergence required. As pointed out in \cite{AL:2023(AppliedNumMath)}, it is easy to see (A) implies (B) by integrating across \eqref{eq.y} and re--arranging: 
	\begin{equation}\label{eq.aveRepy}
		\int_{t-\theta}^t f(s)\,ds = y(t) - y(t-\theta) +\int_{t-\theta}^t y(s)\,ds,
	\end{equation}
	since the righthand side tends to zero as $t\to\infty$ if $y(t)$ does. 
	%
	
	We state results characterising the rate of growth or decay of solutions of \eqref{eq.y}. We allow either for an increasing upper bound, or one which grows or decays to zero subexponentially 
	(exponential or faster--than--exponential decay is not interesting, because solutions of \eqref{eq.unforced}  decay at a power law rate and swamp exponentially small terms).  
	Recall that a function is $\gamma\in C((0,\infty);(0,\infty))$ is subexponential if 
	\begin{equation} \label{eq.psisub}
		\lim_{t\to\infty} \frac{\gamma(t-\theta)}{\gamma(t)}=1, \quad\theta>0.
	\end{equation}
	This is part of a definition of subexponential functions (see e.g., \cite{BGT}). For a comprehensive discussion of properties of this class of functions in the context of differential equations of the type studied in this paper, see \cite{AppLaw2024ODE}. We list some relevant properties to our investigation now. 
	The convergence in \eqref{eq.psisub} is in fact uniform on $\theta$--compacts, so if $\Gamma(t)=\int_{t}^{t+1} \gamma(s)\,ds$, then $\Gamma(t)\sim \gamma(t)$ as $t\to\infty$ and moreover $\Gamma'(t)/\Gamma(t)\to 0$ as $t\to\infty$. By asymptotic integration of $\Gamma'(t)/\Gamma(t)\to 0$ as $t\to\infty$ we get $\log \Gamma(t)/t\to 0$ as $t\to\infty$, whence $\log \gamma(t)/t\to 0$ as $t\to\infty$; thus we get
	$\lim_{t\to\infty} e^{\epsilon t} \gamma(t)=+\infty$ and 
	$\lim_{t\to\infty} e^{-\epsilon t} \gamma(t)=0$ for all $\epsilon>0$. These limits justify the terminology subexponential. 
	Lastly, by using l'H\^opital's rule, the above remarks show,  for any $\alpha>0$, that  
	\begin{equation} \label{eq.subconvexp}
		\int_0^t e^{-\alpha(t-s)}\gamma(s)\,ds \sim \frac{1}{\alpha}\Gamma(t)\sim \frac{1}{\alpha}\gamma(t) \quad t\to \infty.
	\end{equation}
	The next result from \cite{AppLaw2024ODE} characterises of growth or decay rates of the solution \eqref{eq.y}, if the upper bound is monotone or subexponential. The crucial object to study is the function $f_\theta(t):=\int_{(t-\theta)^+}^t f(s)\,ds$, $t\geq 0$, $\theta\in [0,1]$; we compare it to a weight function $\gamma$ via hypotheses such as \eqref{eq.fthetagammabdd},  
	\begin{gather} 
		\label{eq.fthetagamma0}
		\text{
			$f_\theta(t)=o(\gamma(t))$ as $t\to\infty$}; \\
		\label{eq.fthetaequivgamma}
		\text{There exists $\theta'>0$ such that $\limsup_{t\to\infty} |f_\theta(t)|/\gamma(t)>0$}.
	\end{gather} 
	\begin{theorem} \label{lemma.fthetaymonotone}
		Let $f\in C([0,\infty);\mathbb{R})$, $y$ be the unique continuous solution to \eqref{eq.y}. 
		Let either $\gamma$ be (a)  positive, non--decreasing and in $C(0,\infty)$ or (b) subexponential. Then 
		\begin{enumerate}
			\item[(i)] \eqref{eq.fthetagammabdd} $\Longleftrightarrow$ 
			$y(t)=\text{O}(\gamma(t))$, as $t\to\infty$;	
			\item[(ii)] \eqref{eq.fthetagamma0} $\Longleftrightarrow$
			$y(t)=\text{o}(\gamma(t))$, as $t\to\infty$;	
			\item[(iii)] \eqref{eq.fthetagammabdd} and \eqref{eq.fthetaequivgamma}  $\Longleftrightarrow$
			$\limsup_{t\to\infty} |y(t)|/\gamma(t) \in (0,\infty)$.
		\end{enumerate} 
	\end{theorem}
	The hypothesis (A) in the above Theorem requires uniform control of $f_\theta$ for all $\theta\in [0,1]$. The choice of $[0,1]$ is arbitrary; control on an interval of the form $[0,\delta]$ for any choice of $\delta>0$ is equivalent. This holds for all hypotheses of this type, whether the weight function $\gamma$ is subexponential, regularly varying or monotone.

We see that the last part Theorem~\ref{lemma.fthetaymonotone} can be used to characterise the size of solutions of \eqref{eq.y} \textit{exactly} to leading order. Roughly, $y(t)$ is $O(\gamma(t))$ but not $o(\gamma(t))$ as $t\to\infty$ if and only if $f_\theta(t)$ is $O(\gamma(t))$ for all $\theta$, but not $o(\gamma(t))$ for some $\theta$, as $t\to\infty$. 
\vspace{-0.25in}

\section{Asymptotics of deterministic pantograph equation}
It is now easy to characterise boundedness of solutions of the pantograph equation. 
\begin{theorem}\label{thm.detasymptoticbound}
	Suppose that $q\in (0,1)$ and \eqref{eq.unforcedasstable} holds. Let $f\in C([0,\infty);\mathbb{R})$ and $x$ be the unique continuous solution to  \eqref{eq.forced}. Then the following are equivalent
	\begin{enumerate}
		\item[(A)] $y$ is bounded;
		\item[(B)] There exists $F\geq 0$ such that $|\int_t^{t+\theta} f(s)\,ds|\leq F$ for all $\theta\in [0,1]$, $t\geq 0$;
		\item[(C)] $x$ is bounded. 
	\end{enumerate}
\end{theorem}
To prove this, apply part (A) of Theorem~\ref{thm.monotonepert} to $z=x-y$. Since $y$ is bounded, $\varphi$ given by \eqref{def.varphi1} is bounded. Thus $z$ is bounded and hence  $x=z+y$ is bounded, so (A) implies (C). (C) implies (A) by virtue of \eqref{eq.yrepx}. Applying part (i) of Theorem~\ref{lemma.fthetaymonotone} with $\gamma(t)=1$ shows that (A) and (B) are equivalent, completing the proof. 

We note of course that taking the results of Theorems~\ref{thm.detasymptoticstab} and~\ref{thm.detasymptoticbound} together characterise when $x$ is bounded but non--convergent. This will happen if and only if condition (B) of Theorem~\ref{thm.detasymptoticbound} holds, but condition (B) of Theorem~\ref{thm.detasymptoticstab} does not hold. 

A simple situation where this happens is if $f(t)\to f^\ast\neq 0$ as $t\to\infty$, or if $f(t)=\sin(t)$ for $t\geq 0$. A more complicated situation, where $f$ is unbounded, is furnished by a simple modification of the second example after Theorem~\ref{thm.detasymptoticstab}. To give boundedness (but non--convergence) of solutions of the perturbed pantograph equation, we now choose $w_n=1/h_n$, while still taking $h_n>1$ and $h_n\uparrow \infty$ as $n\to\infty$ (once again, therefore, the rate of growth of the running maximum of $|f|$ can be arbitrarily fast in time). In this case, for all $\theta\in [0,1]$ and all $t\geq 0$ we have  
\[
\int_{t}^{t+\theta} f(s)\,ds \leq \int_{t}^{t+1} f(s)\,ds \leq \int_{n}^{n+2} f(s)\,ds = 1,
\]   
where $n$ is the largest integer less than or equal to $t$. Thus condition (B) of Theorem~\ref{thm.detasymptoticbound} holds. On the other hand
\[
\int_{n}^{n+1} f(s)\,ds = \frac{1}{2},
\]
for every integer $n$, so condition (B) of Theorem~\ref{thm.detasymptoticstab} does not hold.

An interesting class of bounded solutions are those which tend to non--trivial limits. The situation when this happens can also be characterised. Once again, we believe this is the first such characterisation of convergent solutions for nonautonomous functional differential equations.  We leave to the interested reader the proof of the following theorem, whose proof is similar to those of Theorems~\ref{thm.detasymptoticstab} and \ref{thm.detasymptoticbound}. Consult the proof of~\cite[Lemma 15.9.2]{GLS} to prove the equivalence of (A) and (B); the equivalence of (A) and (C) can be proven with sole recourse to arguments given here.

\begin{theorem}\label{thm.detasymptoticlimit}
	Suppose that $q\in (0,1)$ and \eqref{eq.unforcedasstable} holds. Let $f\in C([0,\infty);\mathbb{R})$ and $x$ be the unique continuous solution to  \eqref{eq.forced}. Then the following are equivalent
	\begin{enumerate}
		\item[(A)] $y(t)$ tends to a non--trivial finite limit as $t\to\infty$;
		\item[(B)] There exists $F\neq 0$ such that for every $\theta\in [0,1]$ that $\frac{1}{\theta}\int_t^{t+\theta} f(s)\,ds\to F$ as $t\to\infty$ for all $\theta\in [0,1]$;
		\item[(C)] $x(t)$ tends to a non--trivial finite limit as $t\to\infty$.  
	\end{enumerate}
\end{theorem}
It can be easily shown that the limits at infinity of $x$ and $y$ are related to $F$ in condition (B) according to 
\[
\lim_{t\to\infty} x(t)=\frac{F}{1-(b+a)}, \quad  \lim_{t\to\infty} y(t)=F.
\] 
%

We now use Theorem~\ref{thm.growthbounds2} to give rather sharp conditions on the asymptotic behaviour of solutions, without imposing pointwise conditions on the forcing function $f$. We start by giving a theorem for large perturbations. 
\begin{theorem}\label{thm.detlargepert}
	Suppose that $q\in (0,1)$ and \eqref{eq.unforcedasstable} holds. Let $f\in C([0,\infty);\mathbb{R})$ and $x$ be the unique continuous solution to \eqref{eq.forced}. Suppose $\gamma$ obeys either (i) $\gamma\in \text{RV}_\infty(\eta)$ for some $\eta>\kappa$ or (ii) $\gamma$ is positive, increasing and in $C^1$, with $\gamma(t)\to\infty$ as $t\to\infty$. 
	\begin{enumerate}
		\item[(I)]
		\eqref{eq.fthetagammabdd} $\Longleftrightarrow$ 
		$y(t)=\text{O}(\gamma(t))$, as $t\to\infty$
		$\Longleftrightarrow$ 
		$x(t)=\text{O}(\gamma(t))$, as $t\to\infty$. 
		\item[(II)] 
		\eqref{eq.fthetagamma0}  $\Longleftrightarrow$  
		$y(t)=\text{o}(\gamma(t))$, as $t\to\infty$ 
		$\Longleftrightarrow$  
		$x(t)=\text{o}(\gamma(t))$, as $t\to\infty$. 
		\item[(III)] \eqref{eq.fthetagammabdd}, \eqref{eq.fthetaequivgamma} 
		$\Leftrightarrow$ 
		$\limsup_{t\to\infty} |y(t)|/\gamma(t)\in (0,\infty)$
		$\Leftrightarrow$ 
		$\limsup_{t\to\infty} |x(t)|/\gamma(t)\in (0,\infty)$.
	\end{enumerate}
	\end{theorem}
	\begin{proof}
Note part (III) follows from parts (I) and (II). 
For part (I), note that if $\gamma$ is regularly varying, it is also subexponential, so (A) and (B) are equivalent by Theorem~\ref{lemma.fthetaymonotone}. Thus is suffices to prove the equivalence of (B) and (C).

Next, let $z=x-y$ and $\varphi(t)=(1+b)y(t)+ay(qt)$. Then $z$ obeys \eqref{eq.z}. If (B) holds, $|y(t)|\leq C_0\gamma(t)$. Then, as $\gamma$ is regularly varying, $|\varphi(t)|\leq C\gamma(t)$ for some $C>0$. Since $\eta>\kappa$, part (C) of Theorem \ref{thm.growthbounds2} applies, so $z(t)=O(\gamma(t))$ as $t\to\infty$. But since $y(t)=O(\gamma(t))$ and $x=y+z$,  $x(t)=O(\gamma(t))$ as $t\to\infty$, which is (C). 

To show (C) implies (B), note that (C) yields $|x(t)|\leq K\gamma(t)$ for $t\geq 0$. This also gives $|x(qt)|\leq K\gamma(qt)\leq L\gamma(t)$, using the regular variation of $\gamma$, for constants $K$ and $L$. Applying the triangle inequality to \eqref{eq.yrepx}, and using these estimates, we get 
\begin{equation} \label{eq.yrvbound}
	|y(t)|\leq |\xi|e^{-t}+K\gamma(t)+\int_0^t e^{-(t-s)}\left(|1+b|K\gamma(s)+|a|L\gamma(s)\right)\,ds.
\end{equation}
Now, note that there is a $\Gamma$ such that $\gamma\sim \Gamma$ and $t\Gamma'(t)\sim \eta\Gamma(t)$ as $t\to\infty$. Thus, $\Gamma'(t)=o(\Gamma(t))$ as $t\to\infty$ and $e^t\Gamma(t)\to\infty$ as $t\to\infty$. Thus the integral in \eqref{eq.yrvbound} is majorised by some positive constant times $e^{-t}\int_0^t e^s \Gamma(s)\,ds$. 
By L'H\^opital's rule, 
\[
\lim_{t\to\infty} \frac{\int_0^t e^s\gamma(s)\,ds}{e^t \Gamma(t)}
=\lim_{t\to\infty} \frac{e^t \Gamma(t)}{e^t\Gamma(t)+e^t\Gamma'(t)}
=\lim_{t\to\infty} \frac{1}{1+\Gamma'(t)/\Gamma(t)}=1.
\]
Since $\gamma\sim \Gamma$, the integral in \eqref{eq.yrvbound} is $O(\gamma(t))$, and thus $y(t)=O(\gamma(t))$, which is (B).

For part (II) in the regularly varying case, note that if $\gamma$ is regularly varying, it is also subexponential, so (A) and (B) are equivalent by Theorem \ref{lemma.fthetaymonotone}. Thus it suffices to prove the equivalence of (B) and (C). Once again, let $z=x-y$ and $\varphi(t)=(1+b)y(t)+ay(qt)$. Then $z$ obeys \eqref{eq.z} and $\varphi(t)=o(\gamma(t))$ as $t\to\infty$, by hypothesis. Revisiting the proof of part (C) of Theorem~\ref{thm.growthbounds2}, it can be shown in a similar way that  $z(t)=o(\gamma(t))$ as $t\to\infty$. Hence $x(t)=o(\gamma(t))$ as $t\to\infty$, so (B) implies (C). To show that (C) implies (B), note by hypothesis that $|x(t)| <\epsilon \gamma(t)$ as $t\to\infty$ for all $t\geq T$. Using the regular variation, this also gives $|x(qt)|\leq C\epsilon \gamma(t)$ for all $t\geq T/q=T_1$ and some $t$ and $\epsilon$ independent $C>0$. Now take $t\geq T_1>T$, and applying the triangle inequality to \eqref{eq.yrepx}, we have 
\begin{multline*}
	|y(t)|\leq |\xi|e^{-t}+\epsilon \gamma(t)+e^{-t}\left|\int_{0}^{T_1} e^s((1+b)x(s)+ax(qs))\,ds\right| \\
	+ \epsilon \left(|1+b| + |a|C\right) \int_{T_1}^t e^{-(t-s)} \gamma(s)\,ds.
\end{multline*}
Dividing by $\gamma(t)$ and taking the limit as $t\to\infty$ on the righthand side, we get limits $\epsilon$, $0$ and $\epsilon \left(|1+b| + |a|C\right)$ respectively for the three terms, using the fact that $\gamma$ is subexponential in the last case. Thus $\limsup_{t\to\infty} |y(t)|/\gamma(t)\leq \epsilon(1+|1+b| + |a|C)$,
so since $\epsilon$ is arbitrary, we have $y(t)=o(\gamma(t))$ as $t\to\infty$, as needed. 

For part (I) in the monotone case, (A) and (B) are equivalent by Theorem~\ref{lemma.fthetaymonotone}, so if (B) and (C) are equivalent, we are done. Let $z=x-y$ and $\varphi(t)=(1+b)y(t)+ay(qt)$. Then $z$ obeys \eqref{eq.z}. Suppose (B) holds, so $|y(t)|\leq C\gamma(t)$. Then, as $\gamma$ is increasing, $|\varphi(t)|\leq C(|1+b|+|a|)\gamma(t)$. Hence by Theorem \ref{thm.monotonepert}, $z(t)=O(\gamma(t))$. But since $y(t)=O(\gamma(t))$ and $x=y+z$, we have $x(t)=O(\gamma(t))$ as $t\to\infty$, or (B).  
To show (C) implies (B), note from (C) that $|x(t)|\leq K\gamma(t)$ for $t\geq 0$. Applying the triangle inequality to \eqref{eq.yrepx}, using this estimate, and using the monotonicity of $\gamma$, we get 
\begin{align*}
	|y(t)|&\leq |\xi|e^{-t}+K\gamma(t)+\int_0^t e^{-(t-s)}\left(|1+b|K\gamma(s)+|a|K\gamma(qs)\right)\,ds\\
	&\leq |\xi|e^{-t}+K\gamma(t)+\int_0^t e^{-(t-s)}\left(|1+b|K+|a|K\right)\,ds \cdot \gamma(t)\\
	&\leq (|\xi|/\gamma(0)+K+ |1+b|K+|a|K)\gamma(t),
\end{align*}
which is (B). The proof of (II) in the monotone case uses  ideas from parts (I) in the monotone and (II) in the regularly varying case, 
which is left to the reader.
\end{proof}

For smaller perturbations, it is hard to prove as clean a result. If $\gamma$ is regularly varying, and $x(t)=O(\gamma(t))$,  using \eqref{eq.yrepx} yields $y(t)=O(\gamma(t))$ as $t\to\infty$. In particular, take $\gamma(t)=t^\kappa$. However, this over--estimates the size of the smallest perturbation needed to preserve the rate of decay $t^\kappa$ of \eqref{eq.unforced}, since we know that if $y(t)=O(\gamma(t))$ where $\psi(t):=\gamma(t)/t^{\kappa+1}$ is regularly varying with index $-1$, and is integrable, then $x(t)=O(t^\kappa)$ as $t\to\infty$. But the integrability of $\psi$ means that $t\psi(t)\to 0$ as $t\to\infty$, and so the largest perturbations which preserve the rate of decay are $o(t^\kappa)$ (without further restrictions being placed on $y$). 

\begin{theorem} \label{thm.smallstochpert}
Suppose that $q\in (0,1)$ and \eqref{eq.unforcedasstable} holds. Let $f\in C([0,\infty);\mathbb{R})$ and $x$ be the unique continuous solution to \eqref{eq.forced}. Let $\gamma\in \text{RV}_\infty(\eta)$. 
\begin{itemize}
	\item[(i)] Let $y(t)=\text{O}(\gamma(t))$ as $t\to\infty$.
	\begin{enumerate}
		\item[(A)] 
		If $\eta< \kappa$, then $x(t)=O(t^\kappa)$ as $t\to\infty$.
		\item[(B)] If $\eta=\kappa$ and  $t\mapsto \gamma(t)/t^{\kappa+1}$ is in $L^1$, then 
		$x(t)=O(t^\kappa)$ as $t\to\infty$.
		Moreover, if $\limsup_{t\to\infty} |y(t)|/\gamma(t)\in (0,\infty)$, then $\limsup_{t\to\infty} |x(t)|/\gamma(t)>0$.
		\item[(C)] If $\eta=\kappa$ and  $t\mapsto\gamma(t)/t^{\kappa+1}$ is not in $L^1$, then 
		\[
		x(t)=\text{O}\left(t^\kappa\int_1^t \frac{\gamma(s)}{s^{\kappa+1}}\,ds\right), \,t\to\infty.
		\]
		Moreover $\limsup_{t\to\infty} |y(t)|/\gamma(t)>0$ implies  $\limsup_{t\to\infty} |x(t)|/\gamma(t)>0$.  
	\end{enumerate}  
	\item[(ii)] \hspace{0.1in} If $x(t)=\text{O}(\gamma(t))$ as $t\to\infty$, then $y(t)=\text{O}(\gamma(t))$ as $t\to\infty$. Likewise, if $x(t)=\text{o}(\gamma(t))$ as $t\to\infty$, then $y(t)=\text{o}(\gamma(t))$ as $t\to\infty$.   
\end{itemize}
\end{theorem}
Notice that the hypotheses on $y$ required to classify the asymptotic behaviour of $x$ are all equivalent to hypotheses on  $f$, as laid out in Theorem~\ref{thm.detlargepert}. Proofs of upper bounds in part (i) follow by considering $z=x-y$, noting that $y=O(\gamma)$ implies $\varphi=O(\gamma)$ and then applying Theorem \ref{thm.growthbounds2} to $z$. The proofs of lower bounds are by contradiction. Note first that (ii) can be shown by using \eqref{eq.yrepx}. Then, contradiction of the desired conclusion leads to supposing that $x(t)=o(\gamma(t))$. But then $y(t)=o(\gamma(t))$ by \eqref{eq.yrepx}, contradicting the hypothesis that $\limsup_{t\to\infty} |y(t)|/\gamma(t)>0$.

We now make some remarks. First, observe that, taken together, Theorems~\ref{thm.smallstochpert} and \ref{thm.detlargepert} cover a very wide class of perturbation sizes, and are in fact very   comprehensive for practical purposes. 

Second, note that part (A) and (B) of the theorem are in the form of a Hartman--Wintner type--result (see e.g.,  \cite{HW55} for the first results by Hartman and Wintner, \cite[Cor X.16.4]{Hart2002} for a typical result on ordinary equations and for linear functional equations, see Pituk~\cite{Pituk2002}). A typical feature of Hartman--Wintner results is that if perturbation terms obeys some integral smallness condition, then the solution of the perturbed equation $x$ can be related to the solution of the unperturbed problem in the form $x=O(u)$, or perhaps $x/u$ tends to a finite limit; we see therefore that parts (A) and (B) of the theorem have this character. Clearly, Hartman--Wintner results are a strong type of ``asymptotic preservation'' result. 

A weaker, but still very useful asymptotic preservation result, is a Perron-type result (for a typical result for linear, or linearisable, autonomous functional differential equations, see e.g., Pituk~\cite{Pituk2006}). In that case,  we show that some nonlinear function of the solution behaves like a known function of $t$. This allows us ignore dependence on initial data in bounds or limits (which can be hard to quantify explicitly anyway), and enables us to admit more irregularity (or to get away with weaker estimation) in the perturbed terms. For linear equations (or equations that can be linearised) Perron--type results often concentrate on (top) Liapunov exponents i.e. $\limsup_{t\to\infty} \log|x(t)|/t$ or $\lim_{t\to\infty} \log |x(t)|/t$, if the limit exists. We see that the existence of these limits is weaker than knowing that $x(t)\sim ce^{\lambda t}$ as $t\to\infty$ for some $c\neq 0$, and this is the price that must typically be paid in Perron--type results. In our case, because the solution has power--law asymptotic behaviour in the unperturbed case, we consider $\limsup_{t\to\infty} \log |x(t)|/\log t$, noting that for almost all solutions of the unperturbed equation we have 
\[
\limsup_{t\to\infty} \frac{\log |u(t)|}{\log t} =\kappa.
\]
We also notice that this is consistent with classic ideas of polynomial asymptotic stability in stochastic equations, due to Mao~\cite{Mao1992a,Mao1992b} and generalisations to other (in general, non--exponential) decay rates; see e.g.,  
Wu and Hu~\cite{WuHu2012}. In addition, we notice that considering solutions on this logarithmic scale is implicit in the bounds obtained for general delay differential equations in Cermak~\cite{cermak}. 

%
%
%

By focussing on the size of $\log |x(t)|$ in this way, we may characterise solution sizes very cleanly in terms of $y$, for both large and small perturbations, in the following Perron--type theorem. 
\begin{theorem} \label{thm.perronthm}
Suppose that $q\in (0,1)$ and \eqref{eq.unforcedasstable} holds. Let $f\in C([0,\infty);\mathbb{R})$ and $x$ be the unique continuous solution to \eqref{eq.forced}. Let $|\theta(t)|\nearrow  \infty$ as $t\to\infty$, and suppose $\theta(t)/\log t\to \eta\in [\kappa,\infty]$ as $t\to\infty$.  
\begin{enumerate}
	\item[(A)] If $\eta=\kappa$, then  
	\[
	\limsup_{t\to\infty} \frac{\log |y(t)|}{\theta(t)}\leq 1\, 
	\Longleftrightarrow\, \limsup_{t\to\infty} \frac{\log |x(t)|}{\theta(t)}\leq 1.
	\]
	\item[(B)] If $\eta\in (\kappa,\infty]$, then 
	\[
	\limsup_{t\to\infty} \frac{\log |y(t)|}{\theta(t)}=1\, 
	\Longleftrightarrow\, \limsup_{t\to\infty} \frac{\log |x(t)|}{\theta(t)}=1.
	\]
\end{enumerate}
\end{theorem}
Part (A) is equivalent to the more easily--understood statement 
\[
\limsup_{t\to\infty} \frac{\log |y(t)|}{\log t}\leq \kappa\, 
\Longleftrightarrow\, \limsup_{t\to\infty} \frac{\log |x(t)|}{\log t}\leq \kappa,
\] 
and part (B), for finite $\eta>\kappa$,  is equivalent to 
\[
\limsup_{t\to\infty} \frac{\log |y(t)|}{\log t}=\eta\, 
\Longleftrightarrow\, \limsup_{t\to\infty} \frac{\log |x(t)|}{\log t}=\eta,
\]
but we have stated Theorem~\ref{thm.perronthm} in the form above to show unity between the small and large perturbation cases, irrespective of the  solution or perturbation size. 

To prove part (B) of Theorem~\ref{thm.perronthm}, we break it into the cases when $\eta\in (\kappa,\infty)$ and $\eta=\infty$. We prove the former case, leaving the similar proof in the latter case to the interested reader. The reader will immediately note that the following proof requires mainly a re--interpretation or re--working of parts of earlier results, but besides this, no important new analysis is needed. 

If $\eta>\kappa$, the first statement in (B) means that for every $\epsilon>0$ there is a $C_\epsilon>0$ such that $|y(t)|\leq C_\epsilon t^{\eta+\epsilon}$ for all $t\geq 1$. Because $\eta>\kappa$, we can apply Theorem~\ref{thm.detlargepert} to give $|x(t)|\leq D_\epsilon t^{\eta+\epsilon}$ for all $t\geq 1$ and some $D_\epsilon>0$. As a result, we get that $\limsup_{t\to\infty} \log |x(t)|/\log t\leq \eta+\epsilon$. Letting $\epsilon\to 0$, we get that $\limsup_{t\to\infty} \log|x(t)|/\theta(t)\leq 1$. Suppose now, by way of contradiction, that the limsup is strictly less than unity; call it $l$. Then, for any $\epsilon>0$ so small that $(l+\epsilon)(\eta+\epsilon)<\eta$, there exists a $T(\epsilon)>0$ such that we have $\log |x(t)|< (l+\epsilon)\theta(t) < (l+\epsilon)(\eta+\epsilon) \log t$ for $t\geq T(\epsilon)$. Therefore, by Theorem~\ref{thm.detlargepert}, we have that there exists a $C_\epsilon>0$ such that $|y(t)|\leq C_\epsilon t^{(l+\epsilon)(\eta+\epsilon)}$ for all $t\geq 1$. Now, take logarithms, divide by $\log t$ and let $t\to\infty$; and noting that the inequality holds for all $\epsilon>0$ sufficiently small, we may further let $\epsilon\to 0^+$, giving 
\[
\limsup_{t\to\infty} \frac{\log |y(t)|}{\log t}\leq l\eta<\eta,
\]
while by hypothesis, this limsup is $\eta$, yielding the desired contradiction. Hence the forward implication holds. 

For the backward implication, the hypothesis implies that for every $\epsilon>0$ there is a $C_\epsilon>0$ such that $|x(t)|\leq C_\epsilon t^{\eta+\epsilon}$ for all $t\geq 1$. By \eqref{eq.yrepx} this implies that there exists $D_\epsilon>0$ such that $|y(t)|\leq D_\epsilon t^{\eta+\epsilon}$ for $t\geq 1$, and so, by taking logs, the limit as $t\to\infty$, and finally the limit as $\epsilon\to 0$ as before, we have $\limsup_{t\to\infty} \log |y(t)|/\theta(t)\leq  1$. As before, assume by way of contradiction that the limsup $l$ is strictly less than unity. Then for all $\epsilon$ sufficiently small we have 
$l+\epsilon<1$, and $(l+\epsilon)(\eta+\epsilon)<\eta$, and consequently for such an $\epsilon>0$ there is a $T(\epsilon)>0$ such that $\log|y(t)|<(l+\epsilon)\theta(t)< (l+\epsilon)(\eta+\epsilon)\log t$ for all $t\geq T(\epsilon)$.  

In the case that $l\eta<\kappa$, we may choose $\epsilon>0$ so that $(l+\epsilon)(\eta+\epsilon)\leq \kappa$; then by Theorem \ref{thm.smallstochpert}, we have that $x(t) = O(t^\kappa\log t)$, which means that $\limsup_{t\to\infty}\log x(t)/\log t\leq \kappa$. But $\limsup_{t\to\infty} \log |x(t)|/\log t=\eta$, so $\eta\leq \kappa$, which is false, and generates a contradiction in that case. 

In the other case, where  $l\eta\geq \kappa$, for every $\epsilon>0$ we have  $(l+\epsilon)(\eta+\epsilon)>\kappa$; in addition choose $\epsilon$ so small that $l+\epsilon<1$, and $(l+\epsilon)(\eta+\epsilon)<\eta$. Then by  Theorem~\ref{thm.detlargepert} we have that $|x(t)|\leq D_\epsilon t^{(l+\epsilon)(\eta+\epsilon)}$, which reduces to 
\[
\limsup_{t\to\infty} \frac{\log |x(t)|}{\log t}\leq (l+\epsilon)(\eta+\epsilon).
\]
By hypothesis, this limsup is $\eta$, so 
$\eta\leq  (l+\epsilon)(\eta+\epsilon)$. But $(l+\epsilon)(\eta+\epsilon)<\eta$, which generates the desired contradiction. We have thus shown (B) in the case where $l\eta\geq \kappa$.

The proof of part (A), where $\eta=\kappa$, is very similar to the argument given above, and again we leave it to the interested reader to complete.

The clean picture generated by the combination of the Hartman--Wintner type Theorem \ref{thm.smallstochpert}, the large perturbation characterisation Theorem~\ref{thm.detlargepert} and the Perron--type characterisation Theorem \ref{thm.perronthm} is spoiled by the inequality in part (A) of Theorem~\ref{thm.perronthm}, and the converse part (ii) in Theorem \ref{thm.detlargepert}. This converse does not give a result like 
\[
x(t)=O(t^\kappa), \, t\to\infty \Longrightarrow
y(t)=O(\gamma(t)), \,t\to\infty  \text{ where }  t\mapsto \gamma(t)/t^{\kappa+1}\in L^1(0,\infty),
\] 
which, together with Theorem~\ref{thm.smallstochpert}, would essentially characterise the maximal allowable size of perturbations; instead, the best we can do in the present work is to show that  
but rather $x(t)=O(t^\kappa)$ as $t\to\infty$ implies $y(t)=O(t^\kappa)$ as $t\to\infty$. 

Some scrutiny of part (ii) of  Theorem~\ref{thm.smallstochpert} may suggest a reason why a clean converse may be difficult to obtain. It suggests that if a solution of the perturbed equation decays faster than the unperturbed equation \eqref{eq.unforced}, this can only be achieved with a rapidly vanishing perturbation. It is not difficult to manufacture examples of this phenomenon. Let $\mu$ be a positive, continuous and integrable function.
Rescale $\mu$ so that $\int_0^\infty \mu(s)\,ds=1$. Let $x(0)=\xi\neq 0$ and suppose $f$ is given by 
\[
f(t)=\xi\left(-\mu(t)-b\int_t^\infty \mu(s)\,ds -a\int_{qt}^\infty \mu(s)\,ds\right), \quad  t\geq 0.
\] 
Then $x(t)=\xi \int_t^\infty \mu(s)\,ds$  for $t\geq 0$ is the unique continuous solution of \eqref{eq.forced}. If we suppose that $\mu$ tends to zero superexponentially, in the sense that $\lim_{t\to\infty} \log \mu(t)/t=-\infty$, then $x$ tends to zero superexponentially, and so does $f$; moreover, the more rapid the decay in $f$, the more rapid the decay in $x$. Thus, we may have the situation that the equation may be is perturbed by two ``small" perturbations, but that the smaller perturbation gives rise to a more slowly decaying solution. For instance, if $a>0$, $x(0)>0$ and $f(t)\equiv 0$, we have that there exist $0<c_1\leq c_2<+\infty$ such that $c_1t^\kappa \leq x(t)\leq c_2t^\kappa$ for all $t\geq 1$, while the solution in the above example, where there is a non--trivial perturbation, decays superexponentially fast when 
$\log \mu(t)/t\to-\infty$ as $t\to\infty$. 

The reader may also wonder whether the upper bound result in parts (i)(A) and (i)(B) of Theorem~\ref{thm.smallstochpert} can be improved in at least some situations to the stronger asymptotic estimate $x(t)\sim ct^\kappa$ for some $c\neq 0$. One situation where this seems especially plausible is that in which $f$ is positive, and $a>0$, since these conditions promote the positivity of the solution, and would tend to mitigate against the oscillation in the solution that is exhibited when $a<0$ and $f\equiv 0$. However, even in these situations where positivity of the solution is promoted, it is still not always possible to rule out the situation that $x(t)/t^\kappa$ is a bounded function which does not tend to a limit. Therefore, Theorem~\ref{thm.smallstochpert} is also sharp in terms of its conclusions; no better order estimate is available, granted the hypotheses imposed therein.   

To see that we cannot always expect solutions to be asymptotic to $t^\kappa$ as $t\to\infty$, consider the function $x$ defined by  
\[
x(t)=t^\kappa\left\{C+\sin\left(\frac{2\pi}{\log(1/q)}\log t\right)\right\}+Dt^\kappa\int_q^t \psi(s)\,ds, \quad t\geq q,
\]  
where $\psi\in RV_\infty(-1)$, $\psi\in L^1(0,\infty)$, and $C>1$, $D>0$. Let $a>0$, $b<0$ and $-b>a$. Let $\pi(t):=C+\sin(2\pi\log t/\log(1/q))$ for $t\geq q$. Define $f(t)=x'(t)-bx(t)-ax(qt)$ for all $t\geq 1$. Then $x$ solves \eqref{eq.forced} on $t\geq 1$ and $x(t)>0$ for all $t\geq q$. Moreover $x(t)/t^\kappa$ is bounded, but does not tend to a limit as $t\to\infty$. We now show that $f$ is asymptotically positive and satisfies the conditions of part (i)(A) of Theorem~\ref{thm.smallstochpert}. 

Note that $\pi(qt)=\pi(t)$ for all $t\geq 1$, so using $b+aq^\kappa=0$ and the uniform convergence theorem, we have that 
\[
\alpha(t):=-(bx(t)+ax(qt))=|b|Dt^\kappa\int_{qt}^t \psi(s)\,ds\sim |b|D\psi(t)t^{\kappa+1}\log(1/q),\quad t\to\infty.
\]
The function on the righthand side is positive, in $\text{RV}_\infty(\kappa)$, but is $o(t^\kappa)$ as
$t\to\infty$, because the integrability and asymptotic monotonicity of $\psi$ (the latter property follows from the regular variation of $\psi$) imply $t\psi(t)\to 0$ as $t\to\infty$. 

On the other hand, a direct calculation yields
\[
x'(t)=t^{\kappa-1} \left(\kappa\pi(t)+ \frac{2\pi}{\log(1/q)}\cos\left(\frac{2\pi}{\log(1/q)}\log t\right)
+D t\psi(t)+\kappa D \int_q^t \psi(s)\,ds\right).
\]
Note that the term in the brackets is uniformly bounded, since $\pi$ is bounded, $\psi$ is integrable, and $t\psi(t)\to 0$ as $t\to\infty$. Hence $x'(t)=O(t^{\kappa-1})$ as $t\to\infty$. Therefore, as $\alpha$ is asymptotic to a $RV_\infty(\kappa)$ function, $x'(t)=o(\alpha(t))$ as $t\to\infty$. Thus $f(t)=x'(t)+\alpha(t)\sim |b|D\log(1/q)\psi(t) t^{\kappa+1}$ as $t\to\infty$. From this it is easy to show that $y(t)\sim f(t)$ as $t\to\infty$. Therefore, we may take $\gamma(t):=|b|D\log(1/q)\psi(t) t^{\kappa+1}$ in Theorem~\ref{thm.smallstochpert} because it is in $\text{RV}_\infty(\kappa)$. Moreover, since $\psi$ is integrable, 
\[
\int_1^T \frac{\gamma(s)}{s^{\kappa+1}}\,ds
=|b|D\log(1/q)\int_1^T \psi(s)\,ds\to L, \quad T\to\infty,
\]
so although $y$ fulfills the properties of part (i)(B) of  Theorem~\ref{thm.smallstochpert}, with $a>0$ and $f$ positive for all $t$ sufficiently large, $t\to x(t)/t^\kappa$ does not tend to a constant as $t\to\infty$. This demonstrates, as claimed above, that the ``big O" result in Theorem~\ref{thm.smallstochpert} is the best possible in general, given the hypotheses imposed here. This example also hints at the possible asymptotic form of the bounded component of $x(t)/t^\kappa$, under appropriate smallness conditions on $y$, and we propose to investigate this further in later works.   

A characterisation, purely in terms of $f_\theta(t)$, of the situation wherein $\log |y(t)|$ has certain upper rates of growth or decay, has not been presented here. However, it seems that such a result should be available, and would be especially useful when considering Perron--type results for differential equations with general unbounded delay.  

Next, we compare the results of Theorems \ref{thm.detlargepert} and \ref{thm.smallstochpert}  with perturbation theorems of Cermak (see e.g., Theorem 2.3 and Example 4.1 in \cite{cermak}), which, in our view provide the sharpest and most comprehensive estimates for the asymptotic behaviour of perturbed functional differential equations with unbounded delay. He supposes that $f(t)=O(t^\beta)$ and $f'(t)=O(t^{\beta-1})$ as $t\to\infty$. Applying Cermak's results to the pantograph equation, we get  
\begin{align*}
x(t)&=O(t^\kappa), \quad t \to\infty, \quad \text{if $\beta<\kappa$},\\
x(t)&=O(t^\kappa\log t), \quad t \to\infty, \quad\text{if $\beta=\kappa$},\\
x(t)&=O(t^\beta), \quad t \to\infty, \quad\text{if $\beta>\kappa$}.
\end{align*}
First, we observe that our results need neither need pointwise estimates on $f$, nor on $f'$, but merely on $y$, where $y'=-y+f$, and the smoothness of $f$ is not needed in our results. Second, it can readily be the case that $f$ has very bad pointwise behaviour, yet $y$ has good behaviour. For example, suppose that 
\[
f(t)=t^{\beta}+e^{t}\sin(e^{2t}), \quad t\geq 0,
\] 
Clearly, $f(t)=O(e^t)$ and $f'(t)=O(e^{3t})$ as $t\to\infty$, yet $y(t)=O(t^\beta)$ as $t\to\infty$. Thus, applying Theorem~\ref{thm.smallstochpert}, we are able to conclude that 
\[
x(t)=O(t^{\max(\kappa,\beta)}), \quad t\to\infty, \text{for $\beta\neq \kappa$} 
\]
and $x(t)=O(t^\kappa\log t)$ as $t\to\infty$ if $\beta=\kappa$, while Cermak's theorem cannot give any conclusions.

Furthermore, the existing results require an explicit parametric comparison between the solution of the unperturbed equation and perturbation. For example, if $f(t)=O(t^\beta\log t)$ as $t\to\infty$, we cannot apply Theorem 2.3 in \cite{cermak} directly; instead, we would note that $f(t)=O(t^{\beta+\epsilon})$ for every $\epsilon>0$, and hope to obtain the estimate  $f'(t)=O(t^{\beta-1+\epsilon})$ on the derivative. Thus, if $\beta<\kappa$, we get $x(t)=O(t^\kappa)$ and if $\beta>\kappa$ we get $x(t)=O(t^\beta)$, but no prediction can be made in the case $\beta=\kappa$. Using Theorem~\ref{thm.smallstochpert} however,  we get 
\[
x(t)=O\left(t^\kappa\int_1^t s^{\beta-\kappa-1}\log s\,ds\right), \quad t\to\infty.
\]   
The estimates concur with \cite[Theorem 2.3]{cermak} in the cases when $\beta\neq \kappa$; but when $\beta=\kappa$, we get 
\[
x(t)=O(t^\kappa \log^2(t)), \quad t\to\infty,
\]
which is not covered by Cermak's theorem. 

These considerations seem to show that the estimation in the Fundamental Lemma~\ref{lemma.fundamental} may be slightly sharper in characterising the transition from small to large perturbations than the proof of \cite[Theorem 2.3]{cermak}. It is an interesting open question to revisit Cermak's delicate and beautiful proof of \cite[Lemma 3.3]{cermak} for equations with general unbounded delay. This result underlies \cite[Theorem 2.3]{cermak}, and is a close analogue of Lemma~\ref{lemma.fundamental}. This is especially worthwhile, because we have shown by means of the general examples in Theorem~\ref{thm.growthbdssharp} that the estimates in Theorem~\ref{thm.smallstochpert} are, in a certain sense, unimprovable.

Finally, we note that Theorem \ref{thm.detlargepert} above deals with large perturbations; suppose that $y(t)=O(e^{\beta t})$ as $t\to\infty$ for some $\beta>0$. Then this perturbation cannot be put into the framework of \cite[Theorem 2.3]{cermak}. However, by Theorem \ref{thm.detlargepert}, we have that $x(t)=O(e^{\beta t})$ as $t\to\infty$. Furthermore, if it is the case that $\limsup_{t\to\infty} |x(t)|/e^{\beta t}\in (0,\infty)$, this can happen if and only if  $\limsup_{t\to\infty} |y(t)|/e^{\beta t}\in (0,\infty)$. However, stronger conclusions about $f$ are not available; for instance consider 
\[
f(t)=e^{\beta t} + ce^{\theta t}\sin(e^{3\theta t})
\]    
where $\theta>\beta>0$. If $c=0$, then $f(t)=O(e^{\beta t})$ and $y(t)=O(e^{\beta t})$ as $t\to\infty$. If $c\neq 0$, then we still have 
$y(t)=O(e^{\beta t})$ as $t\to\infty$, but $f(t)=O(e^{\theta t})$. Thus, both $f$'s generate the same asymptotic size of $y$, and hence of $x$, but knowing the size of $x$ does not allow conclusions to be drawn about the order of $f$, merely that of $y$, or equivalently, by Theorem 6, the order of $f_\theta$.  

In the case of small perturbations i.e. when $y(t)=O(\gamma(t))$ as $t\to\infty$ and $t\mapsto\gamma(t)/t^{\kappa+1}$ in integrable, we have that $x(t)=O(t^\kappa)$ as $t\to\infty$, or to put it another way, $t\mapsto x(t)/t^\kappa$ is a bounded function. In the unperturbed case, the asymptotic  structure of this bounded component has been found; it is the sum of a log--periodic function and a function which tends to zero as $t\to\infty$ (see \cite{KatoMc}). It is the subject of further work to determine whether these small perturbation conditions are sufficient to establish that the bounded component of the perturbed equation retains a similar (log--periodic) structure. 

Some preliminary ideas as to how this might be achieved are presented in Section 5 and at the end of the paper in Section 7. Space constraints, and the necessity to develop a rate characterisation for related stochastic differential equations (see the next section) preclude more exact analysis and discussion in this work. 

Thus, we see that a correct determination of the exact asymptotic order is likely a first step in understanding the detailed asymptotic structure of solutions. Moreover, it opens up the prospect that the solution \textit{might retain the fine structure of its solutions despite bad pointwise behaviour in forcing functions, irregular stochastic terms, and relatively large forcing terms (as measured by $y$)}. In other words, the refined asymptotic bounds presented in this section can be seen as a critical first step in establishing that the solution structure is robust to relatively large and irregular perturbations.     

\vspace{-0.25in}

\section{Stochastically forced equation and auxiliary SDEs}
In this section, we consider the stochastically perturbed pantograph equations
\begin{equation} \label{eq.stochforced2}
dX(t)=(bX(t)+aX(qt)+f(t))\,dt + \sigma(t)\,dB(t), \quad t\geq 0; \quad X(0)=\xi,
\end{equation}
where once again $a$ and $b$ are real numbers and $q\in (0,1)$. For simplicity, we assume $\xi$ is non--random, $f$ is continuous as before, and   
\begin{equation} \label{eq.sigma}
\sigma\in C([0,\infty);\mathbb{R}).
\end{equation}
Here $B=\{B(t):t\geq 0\}$ is a standard Brownian motion. In order that a solution of \eqref{eq.stochforced2} is well--defined and has appropriate properties, we work on a complete probability space $(\Omega,\mathcal{F},\mathbb{P})$, where Brownian motion $B$ generates its own natural filtration $\mathcal{F}^B(t)=\sigma(\{B(s):0\leq s\leq t\})$ with  $\mathcal{F}^B(t)\subset \mathcal{F}$ for each $t\geq 0$. Under these assumptions, \eqref{eq.stochforced2} is considered a shorthand for the integral equation 
\[
X(t)=\xi+\int_{0}^t \{bX(s)+aX(qs) +f(s)\}\,ds + \int_0^t \sigma(s)\,dB(s), \quad t\geq 0
\]
which has a unique, continuous, and ($\mathcal{F}^B(t)$--) adapted solution on $[0,\infty)$ (see e.g., Mao~\cite{MaoBK}). 

In just the way the auxiliary ordinary differential equation \eqref{eq.y} allows sharp results to be proven for the deterministically forced equation \eqref{eq.forced}, solutions of related auxiliary stochastic (ordinary) differential equations enable us to characterise the asymptotic behaviour of solution of \eqref{eq.stochforced2}. Let $Y_0$ and $Y$ be solutions of the SDEs
\begin{gather} \label{eq.Y0}
dY_0(t)=-Y_0(t)\,dt+\sigma(t)\,dB(t), \quad t\geq 0; \quad Y_0(0)=0,	\\
\label{eq.Y}
dY(t)=(-Y(t)+f(t))\,dt + \sigma(t)\,dB(t), \quad t\geq 0; \quad Y(0)=0.
\end{gather}
Both equations possess unique, continuous, adapted solutions. Recalling that $y$ is the solution of \eqref{eq.y}, the representations of the solutions of \eqref{eq.Y0} and \eqref{eq.Y} are 
\begin{eqnarray} \label{eq.Y0rep}
Y_0(t)&=&e^{-t}\int_0^t e^s \sigma(s)\,dB(s), \quad t\geq 0, \\
\label{eq.Yrep}
Y(t)&=&e^{-t}\int_0^t e^s f(s)\,ds+e^{-t}\int_0^t e^s \sigma(s)\,dB(s) = y(t)+Y_0(t), \quad t\geq 0.
\end{eqnarray} 
We show momentarily that the pathwise asymptotic behaviour of \eqref{eq.Y0} and \eqref{eq.Y} can be characterised very precisely. For now, we show how 
$Y$ can be connected with $X$. Let 
$Z=X-Y$. 
Set $\phi_1(t)=ay(qt)+(1+b)y(t)$ for $t\geq 0$ and 
\begin{equation} \label{eq.phistoch}
\phi(t)= aY(qt)+(1+b)Y(t), \quad t\geq 0
\end{equation}
so $\phi(t)=\phi_0(t)+\phi_1(t)$ where  $\phi_0(t)=aY_0(qt)+(1+b)Y_0(t)$ and  
\begin{equation} \label{eq.Z}
Z'(t)=aX(qt)+bX(t)+Y(t) = aZ(qt)+bZ(t)+\phi(t), \quad t\geq 0.
\end{equation}
Thus, we can apply the machinery of the previous section to $Z$, 
once the asymptotic behaviour of $\phi$ is known. This will be the case, since we will good information about the asymptotic behaviour of $y$ and $Y_0$, on which $\phi$ entirely depends. Furthermore, since 
$X=Z+Y$, the asymptotic behaviour of 
$X$ can be found. 

In particular, suppose $b<0$ and $|a|<|b|$ and that $Y(t)\to 0$ as $t\to\infty$ a.s. Then $\phi(t)\to 0$ as $t\to\infty$ a.s. Applying Theorem \ref{thm.1} part (a), we see from \eqref{eq.Z} that $Z(t)\to 0$ as $t\to\infty$ a.s. Thus, $X(t)=Z(t)+Y(t)\to 0$ as $t\to\infty$ a.s. In the case that $Y$ is a.s. bounded, $\phi$ is bounded a.s., and this leads to $Z$ being bounded a.s., by applying part (b) of Theorem \ref{thm.1} to \eqref{eq.Z}. Hence $X=Z+Y$ is bounded a.s. 


Converses result by writing $Y$ in terms of $X$. 
Write $dX(t)=-X(t)+ f(t)+\{(1+b)X(t)+aX(qt)\}\,dt + \sigma(t)\,dB(t)$. 
By stochastic integration by parts, we get    
\begin{equation}  \label{eq.YreptildeX}
Y(t)=X(t)-\xi e^{-t} - \int_0^t e^{-(t-s)}\{(1+b)X(s)+aX(qs)\}\,ds, \quad t\geq 0.
\end{equation}
From \eqref{eq.YreptildeX}, it is clear that if $X(t)\to 0$ as $t\to\infty$, then $Y(t)\to 0$ as $t\to\infty$ a.s., and likewise, if $X$ is bounded a.s., then $Y$ is bounded a.s.

The above discussion shows the importance of possessing good asymptotic information concerning $Y_0$ and $Y$. Fortunately, some of this is already known for $Y_0$, and using known results, more can be established readily from scratch for $Y$. The following was established concerning $Y_0$ in \cite{ACR:2011(DCDS)}, generalising results from \cite{ChanWilliams:1989}.
\begin{theorem}\label{thm.Y0asy}
Suppose $\sigma$ obeys \eqref{eq.sigma}. Define $S(\epsilon)$ for $\epsilon>0$ by \eqref{eq.S} (if the sum  diverges, we take $S(\epsilon)=+\infty$)
and let $Y_0$ be the unique continuous adapted solution of \eqref{eq.Y0}. 
\begin{enumerate}
	\item[(a)] $S(\epsilon)<+\infty$ for all $\epsilon>0$ $\Longleftrightarrow$ $Y_0(t)\to 0$ as $t\to\infty$ a.s.
	\item[(b)] There exists $\epsilon'>0$ such that $S(\epsilon)<+\infty$ for all $\epsilon>\epsilon'$ 
	$\Longleftrightarrow$ $Y_0$ is bounded a.s. 
	\item[(c)] (There exists $\epsilon'>0$ such that $S(\epsilon)<+\infty$ for all $\epsilon>\epsilon'$ and $S(\epsilon)=+\infty$ for all $\epsilon<\epsilon'$)  $\Longleftrightarrow$ $0<\limsup_{t\to\infty} |Y_0(t)|<+\infty$ a.s. 
	\item[(d)] $S(\epsilon)=+\infty$ for all $\epsilon>0$  $\Longleftrightarrow$ $\limsup_{t\to\infty} |Y_0(t)|=+\infty$ a.s. 
\end{enumerate}
\end{theorem}
It can be shown that $S$ can only obey assumption (A) in part (a), (c) and (d) of the theorem, so solutions of \eqref{eq.Y0} either tend to zero with probability one, are bounded but non--convergent with probability one, or are unbounded with probability one.  

Notice the asymptotic behaviour of $\sigma$, measured through the function $\sigma^2_1(t)$ entirely governs the asymptotic behaviour of $Y_0$.
This mirrors the manner in which the asymptotic decay or boundedness of $f_\theta(t)$ governs the convergence or boundedness of solutions of the deterministic equations \eqref{eq.y} and \eqref{eq.forced}.   

%
It is also shown in \cite{ACR:2011(DCDS)} that if $\sigma^2_1(n)$ is asymptotic to a monotone (non--increasing) sequence, $Y_0(t)\to 0$ as $t\to\infty$ a.s. if and only if $\sigma^2_1(n) \log n \to 0$ as $n\to\infty$. In rough terms, therefore,  $\sigma^2_1(t)$ must tend to zero faster than $1/\log t$ in order for $Y_0(t)$ to tend to zero as $t\to\infty$, a.s. 

We now give a characterisation of the a.s. convergence and boundedness of $Y$. In broad terms, we have convergence to zero (respectively, boundedness) for the ``doubly'' perturbed equation if and only if both the ``singly''
perturbed equations tend to zero (respectively, are bounded).
\begin{theorem} \label{thm.Yto0bddchar}
Let $f,\sigma$ obey \eqref{eq.fcns} and \eqref{eq.sigma}. Let $Y$ be the unique continuous adapted solution to \eqref{eq.Y}, and suppose $S$ is given by \eqref{eq.S}. 
\begin{itemize}
	\item[(a)] The following are equivalent:
	\begin{enumerate}
		\item[(A)] There exists $\Delta>0$ such that $f_\theta(t) \to 0$ as $t\to\infty$ for all $\theta\in [0,\Delta]$, $S(\epsilon)<+\infty$ for all $\epsilon>0$;
		\item[(B)] $Y(t)\to 0$ as $t\to\infty$ a.s. 
	\end{enumerate}
	\item[(b)] The following are equivalent:
	\begin{enumerate}
		\item[(A)] There exists $\Delta>0$, $\bar{F}>0$ and $\epsilon'>0$ such that 
		\[
		\sup_{0\leq \theta\leq \Delta}|f_\theta(t)|\leq \bar{F}, \quad t\geq 0, \quad \text{and}\quad 
		S(\epsilon)<+\infty \quad \epsilon>\epsilon'.
		\]
		\item[(B)] $Y$ is a.s. bounded. 
	\end{enumerate}
\end{itemize} 
\end{theorem}
\begin{proof}
Note $Y(t)=Y_0(t)+y(t)$ for $t\geq 0$, where $y$ is the solution of \eqref{eq.y}. 
For part (a), we prove first  (A) implies (B). The first part of (A) gives $y(t)\to 0$ as $t\to\infty$, by the implication in Theorem~\ref{thm.detasymptoticstab} that (B) implies (A). The second part of (A) gives $Y_0(t)\to 0$ as $t\to\infty$ a.s., by part (a) of Theorem~\ref{thm.Y0asy}. Hence $Y(t)\to 0$ as $t\to\infty$.

To show (B) implies (A), note $Y(t)$ is Gaussian distributed for each $t$, with mean $y(t)$. Consider any sequence $t_n\uparrow \infty$. Then $Y_n:=Y(t_n)$ converges to zero a.s., by (B). But since $Y_n$ is Gaussian, it means that it also converges in distribution, which forces $\mathbb{E}[Y_n]\to 0$ as $n\to\infty$, or $y(t_n)\to 0$ as $n\to\infty$. Suppose now, by way of contradiction, that $\limsup_{t\to\infty} |y(t)|>0$. 
Then there exists $c>0$ and a sequence $t_n\uparrow \infty$ such that $|y(t_n)|>c$ for all $n$ sufficiently large. But this contradicts the fact that (B) implies $y(t_n)\to 0$ as $n\to\infty$ for any divergent sequence $(t_n)$. Hence, we must have $\limsup_{t\to\infty} |y(t)|=0$, or $y(t)\to 0$ as $t\to\infty$. But by the implication in Theorem~\ref{thm.detasymptoticstab} that (A) implies (B), we must have that $\int_t^{t+\theta} f(s)\,ds\to 0$ as $t\to\infty$, which proves the first part of (A). To prove the second part, note we have shown $y(t)\to 0$ as $t\to\infty$, and assumed $Y(t)\to 0$ as $t\to\infty$ a.s. Hence $Y_0(t)=Y(t)-y(t)\to 0$ as $t\to\infty$ a.s. Now apply part (a) of Theorem \ref{thm.Y0asy} to see that  we must have $S(\epsilon)<+\infty$ for all $\epsilon>0$, which is the second part of (A). Thus we have shown that (B) implies (A), establishing the required equivalence to complete the proof of part (a).  

For part (b), we prove first (A) implies (B). Note the first part of (A) implies $y$ is bounded, applying Theorem~\ref{thm.detasymptoticbound}
with  $\gamma(t)\equiv 1$. The second part of (A) implies $Y_0$ is a.s. bounded, by part (b) of Theorem \ref{thm.Y0asy}. Thus $Y=Y_0+y$ is a.s. bounded, giving (B). 

To show that (B) implies (A), it is enough to show that $y$ is bounded; in that case the first part of (A) holds, by the converse half of Theorem~\ref{thm.detasymptoticbound}
with  $\gamma(t)\equiv 1$. On the other hand, if $y$ is bounded, then $Y_0=Y-y$ is a.s. bounded, by hypothesis. Thus, by part (b) of Theorem \ref{thm.Y0asy}, the second part of (A) holds, and hence (B) implies (A). Hence (A) and (B) are equivalent, establishing part (b) of the theorem. 

It remains to show that (B) implies  $y$ is bounded. Let $m$ be any positive, non--decreasing continuous function with $m(t)\to \infty$ as $t\to\infty$ 
and $(t_n)$ any sequence such that $t_n\uparrow \infty$ as $n\to\infty$. 
Now, $Y_n:=Y(t_n)/m(t_n)$ is a sequence of Gaussian random variables with mean $y(t_n)/m(t_n)$, and $Y_n\to 0$ as $n\to\infty$ a.s. since $Y$ is a.s. bounded. But then $Y_n$ converge in distribution, so $y(t_n)/m(t_n)\to 0$ as $n\to\infty$.   

Suppose next, by way of contradiction, that $\limsup_{t\to\infty} |y(t)|=+\infty$. Then $y^\ast(t):=\max_{0\leq s\leq t} |y(s)|\to\infty$ as $t\to\infty$. Pick $m(t)=\sqrt{y^\ast(t)}$, so that $m(t)\to\infty$ as $t\to\infty$, and let $t_n$ be a sequence along which $y^\ast(t) = |y(t)|$. Then 
$|y(t_n)|/m(t_n) = y^\ast(t_n)/\sqrt{y^\ast(t_n)} = \sqrt{y^\ast(t_n)}\to \infty$ as $n\to\infty$, 
which contradicts what was shown at the end of the last paragraph. Hence the supposition $\limsup_{t\to\infty} |y(t)|=+\infty$ is false, and $y$ is bounded, as required. 
\end{proof}
This result, together with the discussion before Theorem~\ref{thm.Y0asy},  gives a characterisation of stability and  boundedness for \eqref{eq.stochforced2}.
\begin{theorem}\label{thm.stochstabtilX1}
Suppose that $q\in (0,1)$ and \eqref{eq.unforcedasstable} holds. Let $f,\sigma\in C([0,\infty);\mathbb{R})$ and $X$ and $Y$ be the unique continuous adapted solutions to \eqref{eq.stochforced2} and \eqref{eq.Y}. Then
\begin{enumerate}
	\item[(i)] Condition (a)(A) of Theorem~\ref{thm.Yto0bddchar}
	$\Longleftrightarrow$ $X(t)\to 0$ as $t\to\infty$ a.s.;
	\item[(ii)] Condition (b)(A) of Theorem~\ref{thm.Yto0bddchar} $\Longleftrightarrow$  $X$ is bounded a.s. 
\end{enumerate}
\end{theorem}
Some remarks are in order; first, Theorem~\ref{thm.Yto0bddchar} is to the best of our knowledge (and remarkably, given the simplicity of the equation) the first pathwise characterisation of asymptotic stability and boundedness for a stochastic differential equation which is subjected to both deterministic and stochastic perturbations. \textit{We consider this an auxiliary result of tremendous importance}. We see straightaway that it can be used to characterise the a.s. convergence and boundedness of the stochastically and deterministically forced pantograph equation, an equation far more complicated that the SDE, revealing that \textit{the stability conditions are independent of the structure of the underlying differential system}. Again, such a stability characterisation for deterministically and stochastically perturbed functional differential equations is, to the best of our knowledge, the first to appear in the literature.  

To see the potential for Theorem~\ref{thm.Yto0bddchar} to have very wide applicability to characterise many other complicated differential systems which are forced state--independently, we give a sketch of the possible direction of research. Suppose, concretely, we have the following stochastic functional differential equation
\[
dX(t)=(F(t,X_t)+f(t))\,dt + \sigma(t)\,dB(t), \quad t\geq 0,
\]    
where $F$ is a functional on $[0,\infty)\times C([0,\infty))$. We view $X$ as a perturbation of the underlying deterministic system $u'(t)=F(t,u_t)$. We use the common notational convention for functional differential equations that for a continuous function $x:\mathbb{R}^+\to \mathbb{R}$, $x_t$ refers to the history of $x$ i.e., to $\{x(s):0\leq s\leq t\}$. Now, consider $Y$ as above, and let $Z:=X-Y$. Then 
\[
Z'(t)=F(t,Z_t+Y_t)+Y(t), \quad t\geq 0.
\]  
Let $F$ satisfy a uniform functional Lipschitz condition in the state i.e., if   for all $\phi$, $\psi\in C(\mathbb{R}^+;\mathbb{R})$, there exists a constant $K>0$ such that 
\[
|F(t,\phi_t)-F(t,\psi_t)|\leq K\max_{0\leq s\leq t}|\phi(s)-\psi(s)|, \quad t\geq 0.
\]
Note that this condition is sufficient to guarantee a unique continuous adapted solution of the original SFDE on $[0,\infty)$ (see e.g., Mao~\cite{MaoBK}). 
Thus, writing $\tilde{f}(t):=F(t,Z_t+Y_t)-F(t,Z_t)+Y(t)$ for $t\geq 0$, we see that $\tilde{f}(t)\to 0$ as $t\to\infty$, $\tilde{f}$ is continuous on $[0,\infty)$, and $Z$ obeys 
\[
Z'(t)=F(t,Z_t)+\tilde{f}(t), \quad t\geq 0.
\] 
Now, if a good perturbation theory exists for $C_0(\mathbb{R}^+;\mathbb{R})$ perturbations of the equation $u'(t)=F(t,u_t)$, we may be able to show that $Z(t)\to 0$ as $t\to\infty$, and hence that $X(t)=Y(t)+Z(t)\to 0$ as $t\to\infty$. 

In the converse direction, stability of the unperturbed functional differential equation means that we should impose the condition 
\[
F(t,\phi_t)\to 0 \quad \text{as $t\to\infty$ for each $\phi\in C_0(\mathbb{R}^+;\mathbb{R})$}.
\]
By our usual arguments, $Y$ can be connected to $X$ via
\[
Y(t)=X(t)-X(0)e^{-t}-\int_{0}^t e^{-(t-s)}\{X(s)+F(s,X_s)\}\,ds, \quad t\geq 0.
\]
Therefore, if $X(t)\to 0$ as $t\to\infty$, the term in curly brackets in the integrand is a continuous function in $s$, and tends to $0$ as $s\to\infty$. Therefore, the integral tends to zero as $t\to\infty$, and hence $Y(t)\to 0$ as $t\to\infty$. From this brief discussion, it can be seen how an effective convergence theory for many ``doubly perturbed'' functional differential equations can be developed. In rough terms, provided the asymptotic stability of the unperturbed is robust to deterministic $C_0$ perturbations, then the convergence of the stochastic FDE system is equivalent to the convergence of the SDE, which has already been characterised in Theorem \ref{thm.Yto0bddchar}. 
\vspace{-0.25in}

\section{Developments and Further Work}
The principal purpose of this paper was to develop necessary and sufficient conditions for the convergence of perturbed solutions of a celebrated unbounded delay differential equation to the  equilibrium solution of the underlying unforced equation. Making these very particular choices, we have demonstrated that the forcing terms which are necessary and sufficient to preserve the asymptotic stability are precisely those needed to give the convergence to equilibrium of simple linear differential equations (whether deterministic or stochastic). 

One reason to focus on the pantograph equation, rather than to consider more general equations, is to exploit the fact that the asymptotic behaviour of the unperturbed pantograph equation is (essentially) completely understood, whereas for equations with general unbounded delay, and finite dimensional equations, for example, complete characterisations of the asymptotic behaviour are not always available. This makes it harder to determine whether gaps in the analysis for perturbed equations are purely the result of shortcomings in the perturbed (or stochastic) analysis, or may instead reflect the incomplete state of knowledge in the unperturbed systems. The abundance of necessary and sufficient conditions for various phenomena in the preceding questions supports this conservative approach.   

On the other hand, pantograph equations represent only a particular example of equations with unbounded delay, and it can be the case that the arguments used to analyse scalar equations can possess special features which make generalisation to higher dimensions difficult, or even impossible. Therefore, in this section, we give  arguments to show that a number of generalisations are possible, despite possibly incomplete knowledge about underlying unperturbed equations, or other important auxiliary equations (e.g., the SDE satisfied by $Y$). Leaving aside questions about growth or decay rates for general problems momentarily, several avenues for generalisation may be readily stated:
\begin{enumerate}
\item[(i)] Equations with general unbounded delay, rather than simply proportional delay;
\item[(ii)] Multi--dimensional equations;
\item[(iii)] Equations with multiplicative, rather than additive noise e.g., equations of the form 
\[
dX(t)=(bX(t)+aX(qt))\,dt + \sigma X(t)\,dB(t), \quad t\geq 0.
\]
\end{enumerate}   
In the following subsections, we sketch details of how the results in the earlier part of the paper can be generalised (rather successfully) to cover these cases. In cases (i) and (ii), the machinery developed in earlier sections can mostly be generalised in reasonably obvious ways (with some new problems emerging), while in case (iii), although the approach of the earlier part of the paper is not useful in the particulars, it  nevertheless suggests an analogous approach to attacking the equations in question which seems to generate some interesting preliminary results.    
\subsection{Equations with general unbounded delay}
For the first question, the framework of the present work is sufficient to provide a positive answer. Consider now the SFDE
\begin{equation}\label{eq.sfde}
dX(t)=(bX(t)+aX(t-\tau(t)))+f(t))\,dt + \sigma(t)\,dB(t), \quad t\geq 0,
\end{equation}
where $\tau$ is a positive continuous function with $\sup_{t\geq 0} t-\tau(t)=:-\bar{\tau}>-\infty$ and $t-\tau(t)\to \infty$ as $t\to\infty$. In this case, we need an initial function $\psi\in C([-\bar{\tau},\infty);\mathbb{R})$ so that if  $X=\psi$ on $[-\bar{\tau},0]$, then there will exist a unique continuous adapted solution of the equation on $[-\bar{\tau},\infty)$. Confining attention to the situation when $X(t)\to 0$ as $t\to\infty$ a.s., we may proceed as above to show that if $X(t)\to 0$ as $t\to\infty$ a.s., then $Y(t)\to 0$ as $t\to\infty$ a.s. We do this by writing $dX(t)=\{-X(t)+(1+b)X(t)+aX(t-\tau(t))+f(t)\}\,dt + \sigma(t)\,dB(t)$, and now, by treating everything apart from $-X(t)$ on the right hand side as a perturbation, using stochastic integration by parts, and rearranging for $Y$, we get  
\begin{equation} \label{eq.YXgeneral}
Y(t)= X(t)-e^{-t}X(0)-\int_{0}^t e^{-(t-s)}\{(1+b)X(s)+aX(s-\tau(s))\}\,ds, \quad t\geq 0.
\end{equation}
Since $t-\tau(t)\to \infty$ as $t\to\infty$, the term in curly brackets tends to zero as $s\to\infty$, so the integral tends to zero, and thus $Y(t)\to 0$ as $t\to\infty$ a.s., as claimed.

To show $Y(t)\to 0$ as $t\to\infty$ implies  $X(t)\to 0$ as $t\to\infty$ a.s., we may define 
$Z(t)=X(t)-Y(t)$ for $t\geq -\bar{\tau}$ (extending $Y$ to be zero on $[-\bar{\tau},0]$), in which case $Z=\psi$ on $[-\bar{\tau},0]$ and $Z$ obeys
\[
Z'(t)=bZ(t)+aZ(t-\tau(t))+\phi(t), \quad t\geq 0
\]  
where $\phi(t)=(1+b)Y(t)+aY(t-\tau(t))$ for $t\geq 0$. Clearly, if $Y(t)\to 0$ as $t\to\infty$ a.s., then $\phi(t)\to 0$ as $t\to\infty$ a.s. Evidently, if $Y(t)\to 0$ as $t\to\infty$ a.s., then $Z(t)\to 0$ as $t\to\infty$ a.s. is a necessary and sufficient condition to ensure that $X(t)\to 0$ as $t\to\infty$ a.s.

Thus, under the condition that the perturbations $f$ and $\sigma$ are ``small'' in such a way that the solutions of the SDE converge to zero, the question of whether deterministic and stochastic perturbations which preserve the convergence of solutions to the equilibrium in the complicated SFDE reduces to whether the deterministically perturbed FDE is convergent if there is a fading perturbation. By virtue of Cermak~\cite{cermak}, it is known that the conditions $b<0$ and $|a|<|b|$ are sufficient for the solution of the unperturbed equation 
\[
u'(t)=bu(t)+au(t-\tau(t)), \quad t\geq 0,
\]
to obey $u(t)\to 0$ as $t\to\infty$. Keeping this condition on $a$ and $b$, we can take the equation for $Z$ and show that $Z$ obeys
\[
D_+|Z(t)|\leq b|Z(t)|+|a||Z(t-\tau(t))|+|\phi(t)|,\quad t\geq 0,
\]
where $D_+$ is the Dini--derivative earlier introduced. Now, by comparison arguments, $|Z(t)| < z(t)$ for $t\geq -\tau$ where $z=|\psi|+1$ on $[-\tau,0]$ and 
\[
z'(t)=bz(t)+|a|z(t-\tau(t))+|\phi(t)|+e^{-t}, \quad t\geq 0.
\]
See Appleby and Buckwar~\cite{AppBuck10}, for example. Finally, it remains to show that $z(t)\to 0$ as $t\to\infty$, and this will give $Z(t)\to 0$ as $t\to\infty$, as needed. 

The proof that $z(t)\to 0$ as $t\to\infty$ can be achieved by means of comparison arguments, showing first that $z$ is bounded, and then by showing that a positive limsup of $z$ leads to a contradiction, and hence to the conclusion that $z(t)\to 0$ as $t\to\infty$. 

The boundedness of $z$ may also be proven by a comparison argument: simply choose a (sufficiently large) constant upper solution to the differential equation for $z$, exploiting the facts that the last terms on the right hand side are uniformly bounded, and that $-b>|a|\geq 0$. The key to the proof of convergence is to view all the terms in $z$, apart from the instantaneous term, as perturbations of the differential equation $v'=bv$,  writing down the variation of constants formula, majorising, and then taking the limsup as $t\to\infty$. From this, it is seen that a positive limsup for $z$ is incompatible with the inequality $-b>|a|\geq 0$.  

Parallel arguments can be developed for the characterisation of bounded solutions. 

Thus, the following theorem can be established.
\begin{theorem}
Let $b<0$, $|a|<|b|$, $\tau\in C([0,\infty);(0,\infty))$ be such that 
$t-\tau(t)\to \infty$ as $t\to\infty$ and $-\bar{\tau}:=\sup_{t\geq 0}(t-\tau(t)) >-\infty$. Let $f$ and $\sigma$ be continuous, and suppose that $\psi\in C([-\bar{\tau},0];\mathbb{R})$. Then there is a unique continuous adapted process $X$ which is the solution of \eqref{eq.sfde} with 
with $X=\psi$ on $[-\bar{\tau},0]$, and moreover:
\begin{enumerate}
	\item[(a)] The following are equivalent:
	\begin{enumerate}
		\item[(A)] $f$ and $\sigma$ obey condition (a)(A) of Theorem \ref{thm.Yto0bddchar};
		\item[(B)] $X(t)\to 0$ as $t\to\infty$ a.s.
	\end{enumerate} 
	\item[(b)] The following are equivalent:
	\begin{enumerate}
		\item[(A)] $f$ and $\sigma$ obey condition (b)(A) of Theorem \ref{thm.Yto0bddchar};
		\item[(B)] $X$ is a.s. bounded. 
	\end{enumerate} 
\end{enumerate}
\end{theorem}
We note that this theorem does not require any smoothness or monotonicity properties on the delay $\tau$, nor on the delayed argument $\text{id}-\tau$. This improves upon the regularity required on $\tau$ in \cite{cermak}, for instance.
\subsection{Decay estimates for general unbounded delay equations}
The question about rates of decay to zero if the solution of \eqref{eq.sfde} demands that a good asymptotic estimate be found to the rate at which $Y(t)\to 0$ as $t\to\infty$. This can then be used to get the rate of decay of $\phi(t)\to 0$ as $t\to\infty$. At this point, it may be hard to use \cite[Theorem 2.3]{cermak} on perturbed equations, since the asymptotic behaviour of $\phi'$ is also needed, and $\phi'$ will not be well--defined for stochastic perturbations, because the non--differentiability of $Y$ will make $\phi$ non--differentiable. There are possible routes to circumvent this problem. One is to invoke a kind of comparison argument in the manner of \cite{AppBuck10} or this paper, but this may not lead to sharp results. Another possibility is to apply a ``double mollification'' to the equation, as follows. Consider $X$, $Y$ and $Z$ as above. Now, define 
$Y_2'(t)=bY_2(t)+\phi(t)$ for $t\geq 0$,
and set $Y_2(t)=0$ for $t\leq 0$. Since $b<0$ and $\phi(t)\to 0$ as $t\to\infty$, we have that $Y_2(t)\to 0$ as $t\to\infty$. Moreover, an upper estimate on the rate of decay of $Y_2$ can be obtained from the rate of decay of $\phi$ (which is known from $Y$). Next, define $Z_2=Z-Y_2$. Then 
\[
Z_2'(t)= bZ_2(t)+aZ_2(t-\tau(t)))+\phi_2(t)
\]
where $\phi_2(t)=aY_2(t-\tau(t))$. Now, provided $\tau$ is continuously differentiable, as it is in Cermak's theory, $\phi_2$ is continuously differentiable, and asymptotic estimates can be obtained for $\phi_2$ and $\phi_2'$ in terms of $Y_2$ and $Y_2'=bY_2+\phi$, all of which can be estimated using pathwise estimates of $Y$. At this point, Cermak's results could be applied off the shelf to give an estimate on $Z_2$. But since $Z=Y_2+Z_2$ and $X=Z+Y$, an estimate on the decay rate of $X$ could be found, again appealing to the decay estimate on $Y$. As a result, we see that a theory that gives good upper bounds for the rate of decay of $X$ can probably be deduced from existing results, provided we can establish good asymptotic estimates for the solution of the SDE. Such a characterisation of the asymptotic behaviour of the SDE will form part of a future work.   

The sharpness of the sufficient conditions generated in this approach could be examined by returning to \eqref{eq.YXgeneral}. If
$X$ has a known a.s. rate of decay, by estimating the terms on the righthand side, an estimate on the necessary rate of decay of $Y$ could be found. 

\subsection{Multidimensional equations with additive noise}
In this subsection, we discuss the multidimensional pantograph equation (rather than equations with general unbounded delay) which is forced both deterministically and stochastically (although, as we mention at the end, most of our argument remains valid for equations with general unbounded delay). We consider $d$--dimensional equations. The analogue of the scalar deterministic perturbation is an $f\in C([0,\infty);\mathbb{R}^d)$ (and we denote its $i$-th component $f_i$). The multidimensional analogue for deterministically perturbed scalar equations is therefore 
\[
x'(t)=Bx(t)+Ax(qt)+f(t), \quad t\geq 0,
\]
where $A$ and $B$ are $d\times d$ real matrices, based on the underlying unforced equation being 
\begin{equation} \label{eq.multidpantou}
u'(t)=Bu(t)+Au(qt), \quad t\geq 0.
\end{equation}
We assume hereinafter that $A$ is not the zero matrix, since in that case all delay differential equations reduce to being  ordinary ones. 

As for the stochastic equation, since there are many dimensions, there may reasonably be many (independent) sources of randomness affecting the dynamics. Focussing on the dynamics of the $i$--th component of the solution vector $X(t)$, if we are to have a state--independent additive noise, we arrive at:
\[
dX_i(t)=\left\{[BX(t)+AX(qt)]_i+f_i(t)\right\}\,dt+\sum_{j=1}^r \sigma_{ij}(t)\,dW_j(t), 
\]
where $W=(W_1,\ldots,W_r)$ is an $r$--dimensional standard Brownian motion (or Wiener process). The $d\times r$ scalar continuous functions $\sigma_{ij}$ can be formed into a $d\times r$ continuous $d\times r$ matrix--values function (its value at time $t$ is denoted $\Sigma(t)$), and then $X$ can be expressed in the differential shorthand 
\begin{equation} \label{eq.stochmultidpanto}
dX(t)=\left\{BX(t)+AX(qt)+f(t)\right\}\,dt+\Sigma(t)\,dW(t), \quad t\geq 0. 
\end{equation}
With $X(0)=\xi\in \mathbb{R}^d$ given, there is a unique continuous adapted process satisfying this equation. 
The question now arises: can we characterise the situation wherein $x(t)\to 0$ as $t\to\infty$ and $X(t)\to 0$ as $t\to\infty$, granted that $u(t)\to 0$ as $t\to\infty$?

As before, we start with a converse result. Let $Y$ be the $d$--dimensional process given by 
\[
dY(t)=\left\{-Y(t)+f(t)\right\}\,dt+\Sigma(t)\,dW(t), \quad t\geq 0,
\] 
with $Y(0)=0\in\mathbb{R}^d$. Then generalising existing arguments to the multidimensional case, we get 
\[
Y(t)= X(t)-e^{-t}X(0)-\int_{0}^t e^{-(t-s)}\{(I_d+B)X(s)+AX(qs)\}\,ds, \quad t\geq 0, 
\] 
where $I_d$ is the $d\times d$ identity matrix. From this, we see that whenever $X(t)\to 0$ as $t\to\infty$ a.s., we have $Y(t)\to 0$ as $t\to\infty$ a.s. This last condition holds provided $Y_i(t)\to 0$ as $t\to\infty$ for each $i$ a.s., and noticing that $Y_i$ obeys the scalar SDE 
\[
dY_i(t)=\{-Y_i(t)+f_i(t)\}\,dt + \sum_{j=1}^r \sigma_{ij}(t)\,dW_j(t),
\]
we see that the equation satisfied by $Y_i$ is of essentially the same form as that satisfied by the solution of \eqref{eq.Y}. In fact, if we define 
\begin{equation} \label{eq.sigmai}
\sigma_i(t):=\left(\sum_{j=1}^r \sigma_{ij}^2(t)\right)^{1/2}, \quad t\geq 0,
\end{equation}
we may use the martingale representation theorem (see e.g., \cite{KS}) to rewrite the continuous martingale 
\[
M_i(t):=\sum_{j=1}^r \int_{0}^t \sigma_{ij}(s)\,dW_j(s), \quad t\geq 0,
\]
as 
\[
M_i(t)=\int_0^t \sigma_i(s)\,d\tilde{B}_i(s), \quad t\geq 0,
\]
where each $\tilde{B}_i$ is a standard one--dimensional Brownian motion (see e.g., Karatzas and Shreve~\cite{KS}). Therefore, $Y_i$ obeys the SDE
\[
dY_i(t)=\{-Y_i(t)+f_i(t)\}\,dt + \sigma_i(t)\,d\tilde{B}_i(t),
\]
which is in precisely the form of \eqref{eq.Y}. Now, define 
$S_i(\epsilon)$ as $S(\epsilon)$ in \eqref{eq.S}, but with $\sigma_i$ in the role of $\sigma$. Since we require $Y_i(t)\to 0$ as $t\to\infty$ for each $i$, a.s., it follows from the converse result (i.e. (B) implies (A)) of part (a) of  Theorem~\ref{thm.Yto0bddchar}, applied to each $Y_i$, that $S_i(\epsilon)<+\infty$ for all $\epsilon>0$ and that for each $f_i$ and each $\theta>0$ the condition 
\[
\int_{t-\theta}^t f_i(s)\,ds \to 0 \quad t\to\infty
\]
holds. This latter condition is nothing but 
\begin{equation} \label{eq.multidimf}
\int_{t-\theta}^t f(s)\,ds\to 0, \quad t\to\infty, \quad \text{for all $\theta>0$}.
\end{equation}
Therefore this condition and $S_i(\epsilon)<+\infty$ for all $\epsilon>0$ and $i\in \{1,\ldots,d\}$ is \textit{necessary} in order that $X(t)\to 0$ as $t\to\infty$ a.s.

Conversely, if each $S_i(\epsilon)<+\infty$ for all $\epsilon>0$ and $i\in \{1,\ldots,d\}$ and $\int_{t-\theta}^t f(s)\,ds\to 0$ as $t\to\infty$, then $Y_i(t)\to 0$ as $t\to\infty$ a.s. for each $i$, and therefore $Y(t)\to 0$ as $t\to\infty$ a.s. As before, set $Z(t):=X(t)-Y(t)$ for $t\geq 0$, and let $\phi(t)=(I_d+B)Y(t)+AY(qt)$ for $t\geq 0$. Then $Z(0)=X(0)$ and $Z$ obeys the perturbed equation 
\begin{equation} \label{eq.pantomultidim}
Z'(t)=BZ(t)+AZ(qt)+\phi(t), \quad t\geq 0.
\end{equation}
Moreover, since $Y(t)\to 0$ as $t\to\infty$ a.s., we have that $\phi(t)\to 0$ as $t\to\infty$ a.s., and since the paths of $Y$ are continuous, $\phi$ is continuous on $[0,\infty)$. We notice that if $Z(t)\to 0$ as $t\to\infty$ a.s., then $X(t)\to 0$ as $t\to\infty$ a.s. 

Therefore, by considering \eqref{eq.pantomultidim} path--by--path (i.e., for each $\omega$ in the a.s. event where $X$ and $Y$ are well--defined, and $Y$ tends to zero), it can be seen that if the \textit{deterministic} property 
\begin{equation} \label{eq.perturbedpropertymultid}
\text{For every $\phi\in C_0([0,\infty);\mathbb{R}^d)$, the solution $Z$ of \eqref{eq.pantomultidim} obeys $Z(t)\to 0$ as $t\to\infty$}
\end{equation}
holds,  
then we have that $Y(t)\to 0$ as $t\to\infty$ a.s. implies $X(t)\to 0$ as $t\to\infty$ a.s. 

We have therefore proven the following result, \textit{which eliminates the need for further analysis of stochastically perturbed equations}, and reduces the study of \eqref{eq.stochmultidpanto} to that of  \eqref{eq.pantomultidim} with deterministic forcing term $\phi\in C_0([0,\infty);\mathbb{R}^d)
$.

\begin{theorem} \label{thm.multi1}
Let $f\in C([0,\infty);\mathbb{R}^d)$ and $\Sigma\in C([0,\infty);\mathbb{R}^{d\times r})$. With $\sigma_{ij}=\Sigma_{ij}$,  define $\sigma_i$ for each $i\in \{1,\ldots,d\}$ according to \eqref{eq.sigmai}. 
Suppose that the property \eqref{eq.perturbedpropertymultid} holds for solutions of \eqref{eq.pantomultidim}. 
Let $X$ be the unique continuous adapted process which obeys \eqref{eq.stochmultidpanto}.
The following are then equivalent:
\begin{enumerate}
	\item[(A)] $f$ obeys \eqref{eq.multidimf} and for each $i$ and $\epsilon>0$ 	\begin{equation} \label{eq.Si}
		S_i(\epsilon):= \sum_{n=1}^\infty \sqrt{\int_{n-1}^n\sigma_i^2(s)\,ds} \exp\left(-\epsilon\int_{n-1}^{n}\sigma_i^2(s)\,ds\right)<+\infty.
	\end{equation}
	\item[(B)] $X(t)\to 0$ as $t\to\infty$ a.s.
\end{enumerate}
\end{theorem} 
Of course, the theorem would be vacuous if property \eqref{eq.perturbedpropertymultid} does not hold. We now explore the property \eqref{eq.perturbedpropertymultid} a little further, and show that it holds under conditions which are, in a certain sense, reasonable multidimensional analogues of \eqref{eq.unforcedasstable}. Since \eqref{eq.perturbedpropertymultid} must hold in the case $\phi(t)=0$ for all $t\geq 0$, this means the all solution $u$ of the unperturbed pantograph equation 
\eqref{eq.multidpantou} must obey $u(t)\to 0$ as $t\to\infty$. Iserles has shown in \cite{Iser} that this is equivalent to 
\begin{equation} \label{eq.iser1}
\text{All eigenvalues of $B$ have negative real parts}
\end{equation}
(a condition which implies $B$ is invertible), together with the spectral radius condition 
\begin{equation} \label{eq.iser2}
\rho(B^{-1}A)<1.
\end{equation}
Therefore, the conditions \eqref{eq.iser1} and \eqref{eq.iser2} should be imposed on $A$ and $B$. However, to the best of our knowledge, no result exists in the literature which states that \eqref{eq.iser1} and \eqref{eq.iser2} are sufficient to ensure that property \eqref{eq.perturbedpropertymultid} holds. It would clearly be of value to prove such a result, since, if it holds, a complete characterisation of convergence of solutions of  perturbed stochastic and deterministic pantograph equations would be achieved. A cursory reading of the work of Iserles in \cite{Iser}, however, does not lead to optimism that his approach could fulfill this easily. The proofs in that paper rely on analyticity of the solutions of the unperturbed equation, and so--called Dirichlet series solutions that can be formed as a consequence of the analyticity. For stochastic equations at least, the forcing function $\phi$ is continuous, but nowhere differentiable, so $Z$ is a $C^1$ function, but not even $C^2$, for instance.     

We show next, however, that for a subclass of $A$ and $B$ obeying properties \eqref{eq.iser1} and \eqref{eq.iser2}, that property \eqref{eq.perturbedpropertymultid} does hold. We retain property \eqref{eq.iser1}, and note that \eqref{eq.iser2} in a certain sense demands that $A$ be small relative to $B$. Our sufficient condition imposes such a smallness condition on $A$ in lieu of \eqref{eq.iser2}. 

It is worth noting that although our argument is not sharp in the pantograph case, it works without modification for general delay; this is appealing, because, to the best of our knowledge, necessary and sufficient conditions such as \eqref{eq.iser1} and \eqref{eq.iser2} are not known which characterise the stability of unperturbed multidimensional equations with unbounded delay. 

To prove property \eqref{eq.perturbedpropertymultid}, we continue to assume \eqref{eq.iser1}; therefore  it follows that there is a unique symmetric, positive definite $d\times d$ matrix $Q$ (with real entries) such that 
\begin{equation} \label{eq.liapB}
B^TQ+QB=-I_d,
\end{equation}
where, here and in the sequel, the superscripted $T$ denotes the transpose of a matrix. Now consider $z(t):=Z(t)^TQZ(t)$ for $t\geq 0$. We see that $z(t)\in \mathbb{R}$ for each $t\geq 0$. Furthermore, since $Q$ is positive definite, $z(t)\geq 0$ for all $t\geq 0$. Our strategy is to derive a perturbed scalar pantograph differential inequality for $z$. Using \eqref{eq.liapB} we get 
\begin{align}
z'(t)
&=(Z(t)^TB^T+Z(qt)^TA^T+\phi(t)^T)QZ(t)
\nonumber\\
&\qquad+Z(t)^T  (QBZ(t)+QAZ(qt)+Q\phi(t))\nonumber\\
\label{eq.zDDEmaster}
&=-\|Z(t)\|^2+2Z(t)^TQAZ(qt) +2\phi(t)^TQZ(t), \quad t\geq 0.
\end{align}
Notice next that the positive definiteness of $Q$ means that 
$c_1\|Z(t)\|^2\leq z(t)\leq c_2\|Z(t)\|^2$ for some $0<c_1\leq c_2$, where $\|\cdot\|$ is the Euclidean norm on $\mathbb{R}^d$ (in fact, $c_1$ and $c_2$ are the minimum and maximum eigenvalues of $Q$, respectively).

Now we estimate the terms on the righthand side of 
\eqref{eq.zDDEmaster}. 
We start with the last term. By the Cauchy--Schwarz inequality, we get
\[
|2\phi(t)^TQZ(t)|
\leq 2\|\phi(t)\|\|QZ(t)\|
\leq \|\phi(t)\|\|Q\|\|Z(t)\|=:2\epsilon_1(t)\|Z(t)\|,
\]
where $\epsilon_1(t):=\|\phi(t)\|\|Q\|$ is continuous, nonnegative and tends to zero as $t\to\infty$, and for any $d\times d$ matrix $M$, $\|M\|$ denotes the matrix norm of $M$ induced from the Euclidean norm.    
Now, recall that $2xy\leq x^2+y^2$ for all $x,y\geq 0$, so, putting $x=\sqrt{\epsilon_1(t)}$ and $y=\sqrt{\epsilon_1(t)}\|Z(t)\|$, we get 
\begin{equation} \label{eq.estt3}
|2\phi(t)^TQZ(t)|\leq
2\epsilon_1(t)\|Z(t)\|
\leq \epsilon_1(t)+\epsilon_1(t)\|Z(t)\|^2.
\end{equation}
For the second term on the righthand side of \eqref{eq.zDDEmaster}, we have 
$
2Z(t)^TQAZ(qt)=2Z(t)^TQ^TAZ(qt)=2 (A^TQZ(t))^T Z(qt)$,
so, setting $\beta:=\|A^TQ\|>0$ and using the  Cauchy--Schwarz inequality once again, we have 
\begin{align*}
|2Z(t)^TQAZ(qt)|&\leq 2\|A^TQZ(t)\|\|Z(qt)\|\leq 2\beta\|Z(t)\|\|Z(qt)\|\\
&\leq \beta\left(\alpha\|Z(t)\|^2+\frac{1}{\alpha}\|Z(qt)\|^2\right),
\end{align*}
where, for arbitrary $\alpha>0$ we have used the inequality $2xy\leq \alpha x^2+y^2/\alpha$ for $x, y\geq 0$ at the last step. 

Hence, using this estimate and \eqref{eq.estt3}, we may return to \eqref{eq.zDDEmaster} to get  
\[
z'(t)\leq -\|Z(t)\|^2+\beta\alpha\|Z(t)\|^2+\frac{\beta}{\alpha}\|Z(qt)\|^2+\epsilon_1(t)\|Z(t)\|^2 + \epsilon_1(t).
\]
Next, we note that $\beta>0$; to see this, assume to the contrary that $\beta=0$. Then $A^TQ=0$. Since $Q$ is positive definite, $Q$ is invertible. This forces $A=0$, which we have ruled out by hypothesis, and we have the desired contradiction. Thus, we have that $\beta>0$ and we may    pick $\alpha:=1/(2\beta)>0$. Then 
\[
z'(t)\leq -\frac{1}{2}\|Z(t)\|^2
+2\beta^2\|Z(qt)\|^2+\epsilon_1(t)\|Z(t)\|^2 + \epsilon_1(t),\quad t\geq 0.
\]
Since $\epsilon_1(t)\to 0$ as $t\to\infty$, for every $\epsilon\in (0,1/2)$ there is $T(\epsilon)>0$ such that for $t\geq T(\epsilon)$, $\epsilon_1(t)<\epsilon$. Hence for $t\geq T(\epsilon)$, we have 
\[
z'(t)\leq -\left(\frac{1}{2}-\epsilon\right)\|Z(t)\|^2
+2\beta^2\|Z(qt)\|^2+ \epsilon_1(t).
\]
Finally, $\|Z(qt)\|^2\leq c_1^{-1}z(qt)$ and $\|Z(t)\|^2\geq c_2^{-1}z(t)$, so for $\epsilon\in (0,1/2)$ we have 
\begin{equation} \label{eq.zpantoineq}
z'(t)\leq -\left(\frac{1}{2}-\epsilon\right)\frac{1}{c_2}z(t)
+\frac{2\beta^2}{c_1}z(qt)+ \epsilon_1(t), \quad t\geq T(\epsilon).
\end{equation}
Since for any $d\times d$ matrix $M$ the identity $\|M\|=\|M^T\|$ holds for the 2--norm, and $Q$ is symmetric, we have $\beta=\|Q^TA\|=\|QA\|$. Now, suppose that 
\begin{equation} \label{eq.stabcondn2}
4\|QA\|^2< \frac{c_1}{c_2},
\end{equation}
or equivalently, $2\beta^2c_2/c_1<1/2$. Choose 
$\epsilon\in (0,1/2-2\beta^2 c_2/c_1)$. Then the coefficient of $z(t)$ in \eqref{eq.zpantoineq} is negative, and in modulus exceeds the coefficient of $z(qt)$. Since $\epsilon_1(t)\to 0$ as $t\to\infty$, we may apply our scalar theory to the inequality \eqref{eq.zpantoineq} to deduce that  $z(t)\to 0$ as $t\to\infty$. But since $\|Z(t)\|^2\leq c_1^{-1}z(t)$, this implies $Z(t)\to 0$ as $t\to\infty$. 

Thus we see the sufficient conditions required to ensure property \eqref{eq.perturbedpropertymultid} are \eqref{eq.iser1} together with \eqref{eq.stabcondn2} in place of \eqref{eq.iser2}. We discuss now the condition \eqref{eq.stabcondn2}, and show that it retains much of the spirit of the condition \eqref{eq.iser2}.

The righthand side of \eqref{eq.stabcondn2} is the minimal eigenvalue of $Q$, divided by the maximal eigenvalue of $Q$, so is a quantity which depends on $B$, but not on $A$. Therefore, as earlier claimed, \eqref{eq.stabcondn2} can be interpreted as a smallness condition on $A$, similar to \eqref{eq.iser2}, because, using the submultiplicative property of the matrix $2$--norm,
and the fact that  $\|Q\|=c_2$ for this norm,  we see that 
\[
4\|A\|^2< \frac{c_1}{c_2^3}
\]
implies \eqref{eq.stabcondn2}. Note also that the condition \eqref{eq.stabcondn2} is ``delay independent'' in the sense that \eqref{eq.stabcondn2} does not depend on $q$; it also shares this property with the condition \eqref{eq.iser2}. 

Using the fact that $Q=\int_0^\infty e^{Bs}e^{B^Ts}\,ds$ and $\int_0^\infty e^{2Bs}\,ds = -\frac{1}{2}B^{-1}$, we see that we can rewrite \eqref{eq.iser2} and \eqref{eq.stabcondn2} as 
\[
\rho\left(\int_0^\infty e^{2Bs}\,ds A \right)<\frac{1}{2},
\quad
\left\|\int_0^\infty e^{Bs} e^{B^Ts}\,ds A\right\|<\frac{1}{2}\sqrt{\frac{\text{minimal eigenvalue of $Q$}}{\text{maximal eigenvalue of $Q$}}},
\]
respectively. Noting that $\rho(M)\leq \|M\|$ for any matrix $M$, we see the close connection between the conditions.

To see further similarities between \eqref{eq.stabcondn2} and \eqref{eq.iser2}, note in the one--dimensional case that we $b:=B<0$ and $2BQ=-1$. Since $c_1/c_2=1$, the condition \eqref{eq.stabcondn2} is equivalent to  $1=\sqrt{c_1/c_2}>2\|QA\|=2QA=-A/B=-a/b$, and hence in the case $d=1$, \eqref{eq.iser1} and \eqref{eq.stabcondn2} are exactly \eqref{eq.unforcedasstable}, which is exactly \eqref{eq.iser1} and \eqref{eq.iser2} in the case $d=1$. 

One other connection between \eqref{eq.iser2} and \eqref{eq.stabcondn2} is worth pointing out. If $B$ is real diagonal, then $Q=-\frac{1}{2}B^{-1}$. Suppose $A$ is symmetric. Then $QA=B^{-1}A$ is symmetric, and so $\|QA\|=\rho(QA)=\rho(B^{-1}A)$ and \eqref{eq.stabcondn2} becomes 
\[
\rho(B^{-1}A) < \frac{\text{minimal eigenvalue of $|B|$}}{\text{maximal eigenvalue of $|B|$}},
\]
a condition very near in form to,  but more restrictive than, condition \eqref{eq.iser2}.

The above calculations therefore lead to the following theorem, which is a corollary of Theorem~\ref{thm.multi1}.
\begin{theorem}  \label{thm.multi2}
Let $f\in C([0,\infty);\mathbb{R}^d)$ and $\Sigma\in C([0,\infty);\mathbb{R}^{d\times r})$. With $\sigma_{ij}=\Sigma_{ij}$,  define $\sigma_i$ for each $i\in \{1,\ldots,d\}$ according to \eqref{eq.sigmai}. 
Suppose that \eqref{eq.iser1} holds, and with $Q$ defined by \eqref{eq.liapB}, the condition \eqref{eq.stabcondn2} holds. 
Let $X$ be the unique continuous adapted process which obeys \eqref{eq.stochmultidpanto}.
Then the following are equivalent:
\begin{enumerate}
	\item[(A)] $f$ obeys \eqref{eq.multidimf} and for each $i$ and $\epsilon>0$ 	\begin{equation} \label{eq.Si}
		S_i(\epsilon):= \sum_{n=1}^\infty \sqrt{\int_{n-1}^n\sigma_i^2(s)\,ds} \exp\left(-\epsilon\int_{n-1}^{n}\sigma_i^2(s)\,ds\right)<+\infty.
	\end{equation}
	\item[(B)] $X(t)\to 0$ as $t\to\infty$ a.s.
\end{enumerate}
\end{theorem} 
As stated above, we believe an optimal result would replace the sufficient condition \eqref{eq.stabcondn2} with \eqref{eq.iser2}. 

We finish with the remarks that the sketched results in this section do not really depend on proportional delay, and therefore should hold for any multidimensional equation with unbounded delay, and also that a result characterising boundedness can be formulated and proved with equal ease. 

\subsection{Equations with state--dependent noise}

There are many studies already present concerning the asymptotic behaviour of stochastic pantograph equations with multiplicative noise, or more generally, state--dependent noise. To date, the results on stability, growth bounds etc tend to give upper estimates, and sufficient conditions for asymptotic stability. The question arises, therefore, whether the techniques of this paper might be helpful in obtaining more precise asymptotic results. 

A direct application of the technique is unlikely to be successful in dealing with equations of the form 
\begin{equation} \label{eq.multipanto}
dX(t)=(bX(t)+aX(qt))\,dt+\sigma X(t)\,dB(t)
\end{equation}
since there is no obvious small perturbation to mollify. However, part of the sharpness achieved in this paper is due to 
a successful decomposition of non--differentiable dynamics of the delay--differential equation into an ordinary non--smooth stochastic differential equation, and a pathwise differentiable stochastic delay differential equation, so an adoption of the that aspect of the philosophy of this paper may be useful. We show briefly here how, in the special case that $a>0$, solutions of \eqref{eq.multipanto} have the following asymptotic behaviour. 
\begin{theorem}
Define $\lambda=b-\sigma^2/2$. Let $a>0$ and $X(0)>0$. Let $X$ be the unique strong (positive) solution of \eqref{eq.multipanto}.
\begin{itemize}
	\item[(a)] If $\lambda>0$, then 
	\[
	\lim_{t\to\infty} \frac{1}{t}\log X(t)=\lambda,\quad \text{a.s.}
	\]
	\item[(b)] If $\lambda\leq 0$, then 
	\[
	\lim_{t\to\infty} \frac{1}{t}\log X(t)=0,\quad \text{a.s.}
	\]
\end{itemize}
\end{theorem}
Therefore, the situation is similar to the case for the deterministic pantograph equation (wherein $\sigma=0$), with subexponential asymptotic behaviour when the underlying instantaneous equation 
\[
dv(t)=bv(t)\,dt+\sigma v(t)\,dB(t)
\]
is asymptotically stable (which occurs when $\lambda<0$). When the underlying instantaneous equation is unstable however (when $\lambda>0$), exponential growth occurs, and does so at a rate $\lambda=b-\sigma^2/2>0$.

This result can be established by means of the following multiplicative decomposition. Let 
$\rho$ be given by $\rho(0)=1$ and 
\[
d\rho(t)=b\rho(t)\,dt + \sigma \rho(t)\,dB(t), \quad t\geq 0.
\]
Then $\rho$, like $Y$, is the solution of a linear SDE, and has an explicit solution 
\[
\rho(t) = e^{(b-\sigma^2/2)t+\sigma B(t)}, \quad t\geq 0.
\]
Since $\rho$ is strictly positive, $1/\rho$ is a well--defined process and obeys the SDE
\[
d\left(\frac{1}{\rho(t)}\right) = \frac{\sigma^2-b}{\rho(t)}\,dt -\frac{\sigma}{\rho(t)}\,dB(t), \quad t\geq 0.
\]
The process $Z=X/\rho$ is likewise 
well--defined, and by stochastic integration by parts obeys 
\[
dZ(t)=X(t)d\left(\frac{1}{\rho(t)}\right)+\frac{1}{\rho(t)}dX(t)+\sigma X(t)\cdot -\frac{\sigma}{\rho(t)}\,dt,
\]
which simplifies to 
\[
dZ(t)=\frac{1}{\rho(t)}aX(qt)\,dt,\quad t\geq 0.
\]
Therefore, we have 
\[
Z(t)=X(0)+\int_0^t a Z(qs) \frac{\rho(qs)}{\rho(s)}\,ds, \quad t\geq 0.
\]
Since the integrand is continuous, $Z$ is in fact differentiable, and we have 
\begin{equation} \label{eq.Zstochapuredelay}
Z'(t)= a \frac{\rho(qt)}{\rho(t)}Z(qt)=:a(t)Z(qt), \quad t\geq 0,
\end{equation}
where $\lambda=b-\sigma^2/2$ and $a(t)=a  e^{\lambda (q - 1) t+\sigma(B(qt)-B(t))}$. 
The strategy now is to study the dynamics of $Z$ pathwise, exploiting the fact that $a(t)$ (or equally $\rho$) is known explicitly in terms of the Brownian motion $B$. Finally, we can recover information about $X$ from $X(t)=\rho(t)Z(t)$, again exploiting knowledge of $\rho$. 

When $\lambda>0$, this technique will work well, since in that case, due to the Strong Law of Large Numbers for standard Brownian motion, 
\[
\lim_{t\to\infty} \frac{1}{t}\log |a(t)| = -\lambda(1-q)<0, \quad \text{a.s.},
\]
so 
$a\in L^1(0,\infty)$ a.s. Then, arguing pathwise, $Z$ tends to a finite limit a.s. (call it $Z^\ast$). Therefore 
\[
X(t)\sim Z^\ast \rho(t) = Z^\ast e^{\lambda t+\sigma B(t)}, \quad  t\to\infty, \quad \text{a.s.}, 
\]
so $X$ grows exponentially (under the assumption that the random limit $Z^\ast$ is non--trivial). We have $Z^\ast>0$ in the case that $X(0)>0$ and $a>0$, where $Z(t)\uparrow Z^\ast\in (0,\infty)$ as $t\to\infty$. 

In the case when $Z^\ast=0$, $Z$ tends to  limit $Z^\ast$ according to $\limsup_{t\to\infty} \log|Z(t)|/t\leq -\lambda$. This can be proved by obtaining an a priori exponential rate of decay on $Z$ from \eqref{eq.Zstochapuredelay}, and then recursively estimating the rate of decay by feeding the exponential estimate back into \eqref{eq.Zstochapuredelay}. From the final estimate on $Z$, we end up with $\limsup_{t\to\infty} \log|X(t)|/t\leq 0$. 

Therefore, in the general case (without sign  restrictions on $a$ and $X(0)$), we have a rather complete picture of the long--run dynamics when $\lambda>0$. Either $Z^\ast\neq 0$ and $X$ inherits the asymptotic behaviour of the $\rho$ (and does so rapidly, because $Z$ converges exponentially fast to $Z^\ast$) or $Z^\ast =0$, and $X$ grows more slowly than any exponential function.  

The case $\lambda<0$ is more analogous to the situation covered in this paper, since when $\sigma=0$, $\lambda=b<0$. In that situation, $a(t)$ still has exponential asymptotic behaviour, and for every $\epsilon\in (0,1-q)$ there is a random variable $T(\epsilon)>0$ such that 
\[
D_+|Z(t)|\leq \left(\frac{-\lambda(1-q-\epsilon)}{1-q}+\epsilon\right)e^{-\lambda(1-q-\epsilon)t}|Z(qt)|, \quad t\geq T(\epsilon).
\]
Now, using comparison techniques, one can show by choosing $\bar{Z}$ sufficiently large (take $\bar{Z}>\max_{s\in [qT,T]} |Z(s)|$) that $|Z(t)|\leq \bar{Z} e^{-\lambda(1-q-\epsilon)/(1-q)t}=:Z_+(t)$ for all $t\geq T$. This is done by showing the $Z_+$ obeys the upper differential inequality satisfied by $|Z|$. From this we rapidly deduce that 
\[
\limsup_{t\to\infty} \frac{1}{t}\log |Z(t)|\leq -\lambda, \quad \text{a.s.}
\]
and since $X(t)=\rho(t)Z(t)$ and 
$\lim_{t\to\infty} \log \rho(t)/t=\lambda$ a.s.
we have that 
\[
\limsup_{t\to\infty} \frac{1}{t}\log |X(t)|\leq 0, \quad \text{a.s.}
\]
so that growth of solutions cannot be exponentially fast. 

In the case when $a>0$ and $X(0)>0$, we can strengthen the result to 
\[
\lim_{t\to\infty} \frac{1}{t}\log|X(t)|=0, \quad\text{a.s.},
\]
so the subexponential dynamics $u$ are inherited by $X$. This can be done by building a lower bound for $Z(t)$ by analogous comparison techniques to those outlined above (this works well, because $Z$ is positive, and indeed, increasing), and prove that 
\[
\liminf_{t\to\infty} \frac{1}{t}\log Z(t)\geq -\lambda, \quad \text{a.s.},
\]
which leads to 
\[
\liminf_{t\to\infty} \frac{1}{t}\log X(t)\geq 0, \quad \text{a.s.}
\]
Combining this with the limsup, we get the claimed zero Liapunov exponent. 

The argument in the case when $\lambda=0$ is similar to that when $\lambda<0$.


\section{Proofs}
\begin{proof}[Proof of Lemma~\ref{lemma.fundamental}]
The proof up to \eqref{eq.Knexplicitestimate} follows closely the line of \cite[Theorem 3(i)]{KatoMc}.
By hypothesis, we rule out the case $a=0$, since this reduces the equation to an ODE. With $a\neq 0$, there exists solutions $k\in \mathbb{C}$ to the equation $aq^k+b=0$. Let $k_0$ be any solution of this equation. Then $\kappa=\Re(k_0)$ is uniquely defined and given by $\kappa=-\log(|b/a|)/\log(1/q)$. Let $c=\log q<0$. Then $e^c=q$. Define $w:(-\infty,\infty)\to\mathbb{C}$ by 
$w(s):=e^{-sk_0} z(e^s)$ for $s\in \mathbb{R}$. Then 
\begin{equation} \label{eq.ws+c}
	w(s+c)=e^{-sk_0} z(qe^s)/q^{k_0},
\end{equation}
and since $z$ is differentiable on $(0,\infty)$, for all $s\in\mathbb{R}$, $w'(s)$ exists and is given by 
\begin{align*}
	w'(s)
	&= e^s b(w(s)-w(s+c)) + e^{s-sk_0}\varphi(e^s)-k_0w(s),
\end{align*}
using \eqref{eq.ws+c} and $aq^{k_0}+b=0$. 
Gathering together the $w(s)$ terms on the righthand side and treating the equation as a perturbed linear ODE, we are led to 
\begin{equation} \label{eq.wvop}
	\frac{d}{ds}\left(w(s)e^{k_0s - be^s}\right)
	=-be^se^{k_0s-be^s}w(s+c)+e^{-be^s}e^s\varphi(e^s).
\end{equation}
Next, for $n\in \mathbb{N}$, define $I_n=[s_0-nc,s_0-(n+1)c]$ (note $c<0$), $\tau_n:=s_0-nc$ and $K_n:=\sup_{s\in I_n} |w(s)|$. 
Now, let $s\in I_{n+1}$ and integrate \eqref{eq.wvop} over $[s_0-(n+1)c,s]$:
\[
w(t)e^{k_0 t - be^t}\big|_{t=\tau_{n+1}}^s
=-b\int_{\tau_{n+1}}^s e^t e^{k_0t-be^t}w(t+c)\,dt + \int_{\tau_{n+1}}^s e^{-be^t}e^t\varphi(e^t)\,dt.	
\]
Now, since $b<0$, $-b=|b|>0$ and $s\in I_{n+1}$, taking the complex modulus on both sides, and using the triangle inequality, we get 
\begin{multline*}
	|w(s)||e^{k_0 s}| e^{|b|e^s} 
	\leq |w(s_0-(n+1)c)||e^{k_0(s_0-(n+1)c)}| e^{|b|e^{s_0-(n+1)c}}
	\\+|b|\int_{s_0-(n+1)c}^s e^t |e^{k_0t}|e^{|b|e^t}|w(t+c)|\,dt
	+\int_{s_0-(n+1)c}^s e^{|b|e^t}e^t|\varphi(e^t)|\,dt.	
\end{multline*}
Since $\text{Re}(k_0)=\kappa$, we get 
\begin{multline*}
	|w(s)| e^{\kappa s+|b|e^s} 
	\leq |w(-(n+1)c)|e^{\kappa(s_0-(n+1)c)} e^{|b|e^{s_0-(n+1)c}}
	\\+|b|\int_{s_0-(n+1)c}^s e^t e^{\kappa t+|b|e^t}|w(t+c)|\,dt
	+\int_{s_0-(n+1)c}^s e^{|b|e^t}e^t|\varphi(e^t)|\,dt.	
\end{multline*}
Since the terms in $w$ are bounded by $K_n$, and the term in $\varphi$ by $\epsilon_{n+1}$, we get 
\begin{multline*}
	|w(s)| e^{\kappa s+|b|e^s} 
	\leq K_n \left(e^{\kappa(s_0-(n+1)c)} e^{|b|e^{s_0-(n+1)c}}
	+\int_{s_0-(n+1)c}^s |b|e^t e^{\kappa t+|b|e^t}\,dt\right)
	\\+|b|\int_{s_0-(n+1)c}^s e^{|b|e^t}e^t\,dt \cdot \frac{\epsilon_{n+1}}{|b|}.	
\end{multline*}
Now evaluate the last integral, and bound it by the antiderivative at the upper limit: 
\begin{multline} \label{eq.est1}
	|w(s)| e^{\kappa s+|b|e^s} 
	\leq K_n \left(e^{\kappa(s_0-(n+1)c)} e^{|b|e^{s_0-(n+1)c}}
	-\int_{s_0-(n+1)c}^s be^t e^{\kappa t+|b|e^t}\,dt\right)
	\\+\frac{\epsilon_{n+1}}{|b|} e^{|b|s}.	
\end{multline}
Next, notice that we may take $s_0$ so large that $\kappa-k+|b|/2 e^{s_0-kc}>0$ for all integers $k\geq 0$, which makes $t\mapsto \kappa t+|b|e^t$ monotone on each $I_n$, for any $n$. We focus on the integral remaining on the righthand side in \eqref{eq.est1}. Integrating by parts gives
\begin{multline*}
	\int_{s_0-(n+1)c}^s b e^t e^{\kappa t+|b|e^t}\,dt
	=\frac{be^t}{\kappa-be^t} e^{\kappa t+|b|e^t}\Big|_{s_0-(n+1)c}^s
	\\+\text{O}\left(e^{\kappa s+|b|e^a}\int_{s_0-(n+1)c}^s \left|\frac{d}{ds} \frac{be^t}{\kappa-be^t}\right|\,dt\right),
\end{multline*}
where the big--O term is uniform in $s$, $\kappa$, $s_0$ and $n$, provided $s\in I_{n+1}$ and the stipulations on $s_0$ hold. The integral can therefore be further estimated according to 
\begin{multline*}
	-\left(1-\frac{\kappa}{be^t}\right)^{-1} e^{\kappa t+|b|e^t}\Big|_{s_0-(n+1)c}^s
	+O\left(\kappa e^{-s_0+(n+1)c} e^{\kappa s+|b|e^s}\right)
	\\= -e^{\kappa t+|b|e^t}\Big|_{s_0-(n+1)c}^s
	+O\left(\kappa e^{-s_0+(n+1)c} e^{\kappa s+|b|e^s}\right).
\end{multline*}
Now, we return to \eqref{eq.est1} with this estimate to get
\begin{multline*} 
	|w(s)| e^{\kappa s+|b|e^s} 
	\leq \frac{\epsilon_{n+1}}{|b|} e^{|b|s} \\
	+ K_n \left(e^{\kappa(s_0-(n+1)c)} e^{|b|e^{s_0-(n+1)c}}
	+e^{\kappa t+|b|e^t}\Big|_{s_0-(n+1)c}^s
	+O\left(\kappa e^{-s_0+(n+1)c} e^{\kappa s+|b|e^s}\right)\right).
\end{multline*}
Thus for $s\in I_{n+1}$ we have 
$|w(s)| \leq K_n \left\{1+O\left(\kappa e^{-s_0+(n+1)c} \right)\right\}
+\epsilon_{n+1} e^{-\kappa s}/|b|$.	Taking the sup on both sides over $I_{n+1}$, and noting that $\kappa<0$ yields
\[
K_{n+1}\leq K_n(1+A e^{nc})+\frac{\epsilon_{n+1}}{|b|} e^{-\kappa (s_0-(n+2)c)}
=K_n(1+A e^{(n+1)c})+B'\epsilon_{n+1}e^{\kappa c n},
\]
where $A>0$ 
and $B'>0$. 
Since $e^{-\kappa c}=\alpha\in (0,1)$, with $\lambda_{n+1}=\epsilon_{n+1}/\alpha^{n+1}$ we have 
$K_{n+1}\leq K_n(1+A e^{nc})+B\lambda_{n+1}$ for $n\geq 0$. Thus for $n\geq 2$ 
\[
K_n\leq \prod_{j=0}^{n-1} (1+Ae^{jc}) \cdot K_0 + B\sum_{j=1}^{n-1} \prod_{l=j}^{n-1} (1+Ae^{lc}) \cdot \lambda_j + \lambda_n.
\]
Let $p_n=\prod_{j=0}^{n-1} (1+Ae^{jc})$ for $n\geq 1$. Then $p_n\to p_\ast\in (1,\infty)$ as $n\to\infty$ and so  
\begin{equation} \label{eq.Knexplicitestimate}
	K_n\leq p_n K_0 + B p_n\sum_{j=1}^n \frac{\lambda_j}{p_j}, \quad n\geq 2. 
\end{equation}
Thus, if $\sum_{n=1}^\infty \lambda_n<+\infty$, then $K_n\leq K^\ast$ for all $n\geq 0$. 
On the other hand, if $\sum_{j=1}^n \lambda_j\to \infty$ as $n\to\infty$, then 
$K_n\leq K^\ast \sum_{j=1}^n \lambda_j$ for all $n\geq 1$, by applying Toeplitz lemma to the sum on the righthand side, proving both parts of the lemma. 

If $\lambda_n$ is summable, $K_n\leq K^\ast$. Since $\kappa<0$,  
$\sup_{t\in [e^{s_0-nc},e^{s_0-nc}e^{-c}]} |z(t)|\leq K^\ast [e^{(s_0-nc)}]^\kappa$. Next, for all $t$ sufficiently large there is a (unique) $n\in \mathbb{N}$ such that 
$e^{s_0-nc}\leq t<e^{s_0-nc}e^{-c}$. Since $\kappa<0$, $t^\kappa\geq e^{\kappa(s_0-nc)}e^{-\kappa c} = \alpha e^{\kappa(s_0-nc)}$. Thus
\[
|z(t)|\leq \sup_{t\in [e^{s_0-nc},e^{s_0-nc}e^{-c}]} |z(t)|\leq 
K^\ast [e^{(s_0-nc)}]^\kappa \leq K^\ast \frac{1}{\alpha} t^\kappa,
\]
as required. In the case $\lambda_n$ is not summable, $K_n\leq K^\ast \sum_{j=1}^n \lambda_j$. Since $\kappa<0$, we have 
\[
e^{-\kappa(s_0-nc)}\sup_{s\in [s_0-nc,s_0-(n+1)c]} |z(e^s)|
\leq \sup_{s\in [s_0-nc,s_0-(n+1)c]]} e^{-\kappa s} |z(e^s)|= K_n,
\]
so 
\[
\sup_{t\in [e^{s_0-nc},e^{s_0-nc}e^{-c}]} |z(t)|\leq K_n \alpha^n 
\left(\frac{e^{-\kappa c}}{\alpha}\right)^n e^{\kappa s_0} 
= K_n \alpha^n  e^{\kappa s_0}
\leq K^{\ast\ast} \alpha^n \sum_{j=1}^n \frac{\epsilon_j}{\alpha^j}.
\]
as required. 
\end{proof}
\begin{proof}[Proof of Theorem \ref{thm.growthbounds2}] 
Write $\varphi(t)=\text{O}(\psi(t)t^{\kappa+1})$. We prove the following subcases:
\begin{enumerate}
	\item[(A)] If $\psi\in \text{RV}_{\infty}(-1)$, $\psi\in L^1(0,\infty)$, then 
	$z(t)=\text{O}(t^\kappa)$ as $t\to\infty$.
	\item[(B)] If $\psi\in \text{RV}_{\infty}(-1)$, $\psi\not\in L^1(0,\infty)$, then 
	$	z(t)=\text{O}\left(t^\kappa\int_0^t \psi(s)\,ds\right)$, $t\to\infty$.
	\item[(C)] If  $\psi\in \text{RV}_{\infty}(\beta)$ for $\beta>-1$, then 
	the estimate in (B) holds. Equivalently, if $\varphi=\text{O}(\gamma)$ where $\gamma\in \text{RV}_\infty(\eta)$ and $\eta>\kappa$, then $z(t)=\text{O}(\gamma(t))$ as $t\to\infty$.
	\item[(D)] If  $\psi\in \text{RV}_{\infty}(\beta)$ for $\beta<-1$, then 
	the estimate in (A) holds. Equivalently, if $\varphi=\text{O}(\gamma)$ where $\gamma\in \text{RV}_\infty(\eta)$ and $\eta<\kappa$, then $z(t)=\text{O}(t^\kappa)$ as $t\to\infty$.  
\end{enumerate}	

Since $t\mapsto t^{\kappa+1}\psi(t)$ is regularly varying, and an estimate of the form $|\varphi(t)|\leq Kt^{\kappa+1}\psi(t)$ holds, by the uniform convergence theorem for regularly varying functions, $\epsilon_n = \sup_{s\in [e^{s_0-nc},e^{s_0-nc}e^{-c}]} |\varphi(s)|
\leq \bar{K} e^{(s_0-nc)(\kappa+1)} \psi(e^{s_0-nc})$. 
Since $\psi$ is regularly varying, we get  
$\epsilon_n\leq K^\ast  e^{-nc(\kappa+1)} \psi(e^{-nc})$, so $\lambda_n = \epsilon_n/\alpha^n \leq K^\ast \psi(e^{-nc}) e^{-nc}$. 
By the  uniform convergence theorem for regularly varying functions
\[
\frac{\int_x^{xe^{-c}} \psi(s)\,ds}{\psi(x)x} = \int_1^{e^{-c}} \frac{\psi(\eta x)}{\psi(x)}\,d\eta \to \int_1^{e^{-c}} \eta^\beta\,d\eta =: k>0, \quad x\to\infty,
\]
and so we have that 
\begin{equation} \label{eq.asyconvertsumint}
	e^{-nc} \psi(e^{-nc}) \sim \frac{1}{k} \int_{e^{-nc}}^{e^{-(n+1)c}} \psi(s)\,ds, \quad n\to\infty,
\end{equation}

For part (A), $\psi$ is integrable, so \eqref{eq.asyconvertsumint} shows that $(e^{-nc} \psi(e^{-nc}))$ is summable, and hence $(\lambda_n)$ is summable. Hence by the fundamental lemma, $z(t)=\text{O}(t^\kappa)$, $t\to\infty$. 

For part (D), since $\beta<-1$, $\psi$ is integrable, and by the same argument $z(t)=\text{O}(t^\kappa)$, $t\to\infty$. The second half of the claim follows by making the identification $\psi=\gamma/t^{\kappa+1}$, and since $\eta<\kappa$, $\psi\in \text{RV}_\infty(\beta)$ with $\beta=\eta-\kappa-1<-1$. 

For (B) and (C), $\psi$ is not integrable by hypothesis. Using \eqref{eq.asyconvertsumint}, we have  
\[
\sum_{j=1}^n \lambda_j \leq K^\ast \sum_{j=1}^n \psi(e^{-jc}) e^{-jc}
\leq K^{\ast \ast} \sum_{j=1}^n \frac{1}{k} \int_{e^{-jc}}^{e^{-(j+1)c}} \psi(s)\,ds
\leq \bar{K}  \int_{1}^{e^{-(n+1)c}} \psi(s)\,ds.
\]
Now, since $p_n>1$ and $p_n\to p^\ast$ as $n\to\infty$, by \eqref{eq.Knexplicitestimate}, we get 
\begin{multline*}
	\sup_{t\in [e^{s_0-nc},e^{s_0-nc}e^{-c}]}  |z(t)| \leq \alpha^n K_n
	\leq \alpha^n\left(p_n K_0 + p_n \sum_{j=1}^n \lambda_j/p_j\right)\\
	\leq C_1 \alpha^n + C_2 \alpha^n \sum_{j=1}^n \lambda_j
	\leq C_1 \alpha^n + C_2 \alpha^n \bar{K}  \int_{1}^{e^{-(n+1)c}} \psi(s)\,ds.
\end{multline*}
Since the integral diverges, we have the consolidated estimate
\[
\sup_{t\in [e^{s_0-nc},e^{s_0-nc}e^{-c}]}  |z(t)| \leq C_3 \alpha^n  \int_{1}^{e^{-(n+1)c}} \psi(s)\,ds.
\]
Finally, take $t\in [e^{s_0-nc},e^{s_0-(n+1)c}]$. Then $e^{-nc} e^{-c} \leq e^{-s_0}e^{-c}t$. Since $\psi$ is regularly varying, and $\psi$ is not integrable, $\int \psi$ is also regularly varying. Hence
\[
\int_1^{e^{-(n+1)c}} \psi(s)\,ds \leq \int_1^{e^{-c} e^{-s_0}t} \psi(s)\,ds
\leq C_4  \int_1^{t} \psi(s)\,ds.
\]
On the other hand, $e^{-nc}\leq  e^{-s_0} t$, so $\alpha^n=e^{-nc\kappa}\geq  e^{-\kappa s_0} t^\kappa$. Thus
\[
|z(t)|\leq 	\sup_{t\in [e^{s_0-nc},e^{s_0-nc}e^{-c}]}  |z(t)| \leq C_3 \cdot e^{-\kappa s_0} t^\kappa \cdot C_4  \int_1^{t} \psi(s)\,ds,
\]
as required. In part (B), and for the first part of the claim in (C), this suffices. 

For the second claim in (C), we have $\varphi=\text{O}(\gamma)$, with $\gamma\in \text{RV}_\infty(\eta)$ and $\eta>\kappa$. This allows us to take  $\psi(t)=\gamma(t)/t^{\kappa+1}\in \text{RV}_\infty(\beta)$ with $\beta=\eta-\kappa-1>-1$, so the hypotheses in the first part of the claim apply. Further, by the regular variation of $\psi$,
\[
\int_0^t \psi(s)\,ds \sim \frac{1}{\beta+1} t \psi(t) = \frac{1}{\eta-\kappa} \frac{\gamma(t)}{t^\kappa}, \quad t\to\infty,
\]  
and so applying the first part of the claim gives $z(t)=\text{O}(\gamma(t))$, $t\to\infty$. 

The conclusions of the different parts of the theorem can be combined. Parts (A) and (B) can be combined: let $\varphi=\text{O}(\gamma)$ where $\gamma \in \text{RV}_\infty(\kappa)$. Then $\psi(t):=\gamma(t)/t^{\kappa+1}$ for $t\geq 1$ is in $\text{RV}_\infty(-1)$, and $\varphi(t)=\text{O}(t^{\kappa+1}\psi(t))$. Then, in (A) and (B), we have 
\[
z(t)=\text{O}\left(t^\kappa\int_1^t \psi(s)\,ds\right) = \text{O}\left(t^\kappa\int_1^t \frac{\gamma(s)}{s^{\kappa+1}}\,ds\right), \quad t\to\infty.
\] 
For part (C), $\eta>\kappa$, so $\psi$  
is in $\text{RV}_\infty(\eta-\kappa-1)$. Since $\eta-\kappa-1>0$, the function in the big O--estimate is asymptotic to $\gamma(t)/(\eta-\kappa)$ as $t\to\infty$,  
so by part (C), 
$z(t)=\text{O}(\gamma(t)) = \text{O}\left(t^\kappa\int_1^t \frac{\gamma(s)}{s^{\kappa+1}}\,ds\right)$ as $t\to\infty$,
which agrees with (A) and (B). Finally, for (D), where $\eta<\kappa$, $\psi \in \text{RV}_\infty(\eta-\kappa-1)$. Since $\eta-\kappa-1<-1$, $\psi \in L^1(1,\infty)$, and so by (D)
\[
z(t)=\text{O}(t^\kappa) = \text{O}\left(t^\kappa\int_1^t \frac{\gamma(s)}{s^{\kappa+1}}\,ds\right), \quad t\to\infty.
\] 
This allows the desired consolidation of the statement of the theorem.
\end{proof}

\section{Conclusions}

The paper considers perturbations of the asymptotically stable deterministic pantograph equation 
\[
u'(t)=bu(t)+au(qt)
\]
where $a\neq 0$. As is well--known, all solutions of this equation tend to zero if and only if $b<0$ and $|a|<|b|$. We thus presume these conditions hold, and ask what is the asymptotic behaviour of forced equations 
\begin{align*}
x'(t)&=bx(t)+ax(qt)+f(t), \\
dX(t)&=bX(t)+aX(qt)+f(t)\,dt +\sigma(t)\,dB(t).
\end{align*}  
We can characterise, in terms of $f$ and $\sigma$, necessary and sufficient conditions such that solutions are bounded or convergent to zero, assuming only that the underlying equation has stable solutions. These conditions depend on the time averages of $f$ and $\sigma$ over moving compact intervals, but do not depend on the parameters ($a$, $b$ and $q$) of the underlying unperturbed equation.
The conditions allow $f$ and $\sigma$ to behave very badly on a pointwise basis, while being satisfied by often--assumed stability conditions such as $f(t)\to 0$ or $f\in L^1(0,\infty)$. The conditions on the stochastic perturbation are more delicate; the a.s. stability or a.s. boundedness is captured by the convergence or divergence of an infinite sum which depends on a parameter. This condition is known to characterise the a.s. stability or a.s. boundedness of solutions of the stochastic ordinary differential equation
\[
dY_0(t)=-Y_0(t)\,dt + \sigma(t)\,dB(t)
\]
a result proven in an earlier work. The necessary and sufficient conditions are a little opaque, but a very good sufficient condition is that 
\begin{equation} \label{eq.sig2nlognto0}
\sigma_1^2(n) \log n\to 0, \quad n\to\infty;
\end{equation}
where $\sigma_1^2(n)=\int_{n-1}^{n} \sigma^2(s)\,ds$; 
indeed, if $\sigma^2_n$ is decreasing, \eqref{eq.sig2nlognto0} is necessary, so once again we see that it is the average behaviour of $\sigma^2$ over moving compact intervals that determines the asymptotic behaviour. 

A key observation in the entire analysis is that the stability of the deterministic functional differential equation is inherited from the ordinary differential equation 
\[
y'=-y+f
\] 
and the SFDE from $Y=Y_0+y$, which solves a ``doubly perturbed'' SDE. It turns out $Y$ tends to zero if and only if both $Y_0$ and $y$ are, while $Y$ is bounded if and only if $Y_0$ and $y$ are.   

As well as characterising boundedness and convergence, the approach enables us to determine very sharp conditions describing the rate of growth or decay of deterministic equations. If the solution is to grow, or has a growing running maximum, the rate of growth of its running maximum is that of $y$, and this can be characterised entirely through the time averages $\int_{t-\theta}^t f(s)\,ds$, which has the same asymptotic order of $y$. In terms of estimating the growth order, and characterising when certain growth orders of solutions occur, these results are unimprovable. Crucially, they show that relying on pointwise estimates on $f$ may give very poor and misleading asymptotic information. 

When solutions tend to zero, they do so slowly in most cases, since the unperturbed equation has solutions which tend to zero at the rate $t^\kappa$, where $\kappa<0$ is known in terms of $a$, $b$ and $q$. Existing analysis for general equations with unbounded delay, applied to the pantograph equation, gives estimates of $x(t)=O(f(t))$ when $f(t)=O(t^\beta)$ and $f'(t)=O(t^{\beta-1})$, and $\beta>\kappa$, and $x(t)=O(t^\kappa)$ when $f(t)=O(t^\beta)$ for $\beta<\kappa$. We show, in contrast, that when $y(t)=O(\gamma(t))$ and $\gamma$ is regularly varying, then 
\[
x(t)=O\left(t^\kappa \int_1^t \frac{\gamma(s)}{s^{\kappa+1}}\,ds\right),\quad t\geq 0
\]     
which shows that neither pointwise nor derivative bounds are needed on $f$. Moreover, the estimate we establish is sharper in some cases than existing estimates, is never weaker than existing estimates, and is unimprovable in at least some situations (for equations in which $a>0$). 

The rate of growth of $y$ plays a determining role the asymtotic behaviour when $y$ dominates the solution of the unperturbed equation. For instance, if $\gamma$ is regularly varying with index greater than $\kappa$, or $\gamma$ is increasing, then $x(t)=o(\gamma(t))$ if and only if $y(t)=o(\gamma(t))$, $x(t)=O(\gamma(t))$ if and only if $y(t)=O(\gamma(t))$ and $\limsup_{t\to\infty} |x(t)|/\gamma(t)\in (0,\infty)$ if and only if $\limsup_{t\to\infty}|y(t)|/\gamma(t)\in (0,\infty)$. 

We have sufficient conditions which describe when the rate of decay of the underlying pantograph equation is preserved. If $y(t)=O(\gamma(t))$, $\gamma$ is regularly varying, and $t\mapsto\gamma(t)/t^{\kappa+1}$ is integrable, then $x(t)=O(t^\kappa)$ as $t\to\infty$. These conditions seem, in certain cases at least, unimprovable, since we have shown by means of general examples, that if the integrability condition is relaxed, the unperturbed rate of decay will not be preserved. Good necessary conditions prove for the moment, more elusive. We know that if $x(t)=O(t^\kappa)$, then $y(t)=O(t^\kappa)$; but taking the sufficient condition above into account, it seems that the maximal allowable rate of growth of $y$ would be such that $t\mapsto |y(t)|/t^{\kappa+1}\in L^1(0,\infty)$, a condition satisfied in the case $y(t)=O(t^\kappa\{\log t\log\log t(\log\log \log t)^2\}^{-1})$, for instance. Thus, our best converse result in the case of a small perturbation is still suboptimal. 

Although it might seem strange that we can characterise asymptotic bounds for ``large'' perturbations, but not ``small'' ones, we saw by means of an example that it is possible to arrange for the solution of the perturbed equation to decay at an arbitrarily fast rate, for a very carefully chosen and very rapidly decaying forcing term. Therefore, there is potentially great sensitivity in the rate of decay of solutions when there are small forcing terms, and this probably explains in part why it is hard to prove a very refined converse. As a first step though, our converse is not unacceptable: the above discussion suggests that our estimate on the critical $y$ is only off by a factor which grows at around the (very slow) rate of $\log t$ as $t\to\infty$.   

In this paper, we have not attempted to give quantitative asymptotic estimates for the solutions of stochastic equations. In part, this is because a characterisation of the rates of decay and growth of solutions of \eqref{eq.Y} is needed, and this task forms part of work in progress. However, once the task of characterising the asymptotic rates of growth and decay of $Y$ is completed, the analysis in this paper (see Theorems~\ref{thm.smallstochpert} and \ref{thm.detlargepert}) should be reusable with minor modifications, with Theorems~\ref{thm.growthbounds2} and \ref{thm.monotonepert} applicable to $Z=X-Y$, with the estimate on $\varphi$ being read off from $Y$.

There are two other directions where the results of this paper could readily be developed. One is in the direction of ``unstable'' equations, which is to say equations where the condition \eqref{eq.unforcedasstable} do not hold, and the unperturbed equation has unbounded solutions. The case when $b>0$ should be more straightforward, since it that case, solutions of the unperturbed equation grow like $e^{bt}$, so one would hope that the solution of the perturbed equation would grow at the faster rate between $e^{bt}$ and $\gamma$, where $\gamma$ describes the rate of growth of $y$.

The case when $b<0$, and $|a|>|b|$ is more interesting, since in that case, the unperturbed equation exhibits power law growth at rate $t^\kappa$; now $\kappa>0$. In this situation, we expect analogues of Theorems~\ref{thm.detlargepert} and \ref{thm.smallstochpert} to hold. Revisiting the proof of the fundamental Lemma~\ref{lemma.fundamental} reveals that the estimates hold up to \eqref{eq.Knexplicitestimate}, with perhaps some cosmetic changes needed earlier. Therefore it is probable that we will be able to estimate as in the proof of Theorem~\ref{thm.growthbounds2}. As a result, we expect a theorem analogous to Theorems~\ref{thm.detlargepert} to hold if $y(t)=O(\gamma(t))$ and  $\gamma$ is regularly varying with index greater than $\kappa$ (or $\gamma$ is increasing faster than a regularly varying function), while a theorem
analogous to Theorem~\ref{thm.smallstochpert} should hold if $\gamma$ has index less than or equal to $\kappa$. 

In Section 5, we commented that generalisations to equations with unbounded delay are possible, and we were able to sketch convergence and boundedness characterisation results directly. However, if growth bounds on solutions are required, we noted but that further work would be required, both on stochastic ordinary differential equations, as well as the analysis of Cermak in the deterministic case in order to sharpen estimates. Extensions of the convergence and boundedness results to the finite dimensional case are feasible too, but at the cost of stronger sufficient conditions being imposed on the structure of the unperturbed equation. We conjecture that a perturbation theorem which states that property \eqref{eq.perturbedpropertymultid} holds under the conditions \eqref{eq.iser1} and \eqref{eq.iser2} is true, but do not have immediate ideas how to proceed with it. The situation wherein there is multiplicative noise has been studied more extensively in the literature, but very refined results (i.e., results which classify the dynamics) appear to be lacking. However, by using the spirit of the approach taken here (decomposition of the solution into the solution of a stochastic ordinary differential equation, and the $C^1$ solution of a random functional differential equation), we have produced a result in the spirit of the very first results of Kato and McLeod: exponential asymptotic behaviour when the underlying non--delay equation is unstable, and subexponential asymptotic behaviour when the underlying non--delay equation is asymptotically stable. It is appealing that such progress is made for an equation with multiplicative noise by a multiplicative decomposition, since the bulk of this paper deals with additive noise, and is (we believe) successfully tackled by an additive decomposition of the solution. Further exploration of the multiplicative decomposition proposed here warrants a deeper investigation.     

We finish with a remark that when perturbations are ``small'', the solution $x$ of the perturbed pantograph equation is $O(t^\kappa)$. But it is well--known that solutions of the unperturbed equation exhibit log--periodic oscillation in the case that $a<0$. With further careful analysis, it may be possible to show, for sufficiently small forcing terms, that the perturbed equation inherits this oscillatory asymptotic behaviour, even in the stochastic case. Such results would require developing analogues of Kato and McLeod's arguments. Since some of their analysis involves repeated differentiation of the equation, the double mollification ideas (or even higher order mollifications) sketched in Section 5.2 will be necessary if the results are not to require estimates of $f$ and its derivatives in the deterministic case; indeed, such mollification would be necessary if one wished to prove more refined asymptotic results for stochastic equations, as suggested in Section 5.2.

\bibliographystyle{unsrt}


\end{document}